\documentclass[10pt,a4paper]{article}

\usepackage{latexsym,amsfonts,amsmath,amssymb,mathrsfs,url,color,amsthm}
\usepackage{graphicx}

\newtheorem{theorem}{Theorem}
\newtheorem{lemma}{Lemma}

\newtheorem{remark}{Remark}
\newtheorem{proposition}{Proposition}
\newtheorem{definition}{Definition}

\newcommand{\rd}{\mathrm{d}}
\newcommand{\bse}{\boldsymbol{e}}
\newcommand{\bsk}{\boldsymbol{k}}
\newcommand{\bsl}{\boldsymbol{l}}
\newcommand{\bsq}{\boldsymbol{q}}
\newcommand{\bsr}{\boldsymbol{r}}
\newcommand{\bsw}{\boldsymbol{w}}
\newcommand{\bsx}{\boldsymbol{x}}
\newcommand{\bsy}{\boldsymbol{y}}
\newcommand{\bsz}{\boldsymbol{z}}
\newcommand{\bszero}{\boldsymbol{0}}
\newcommand{\bsone}{\boldsymbol{1}}
\newcommand{\bsalpha}{\boldsymbol{\alpha}}
\newcommand{\bsgamma}{\boldsymbol{\gamma}}
\newcommand{\bsmu}{\boldsymbol{\mu}}
\newcommand{\CC}{\mathbb{C}}

\newcommand{\NN}{\mathbb{N}}
\newcommand{\RR}{\mathbb{R}}
\newcommand{\ZZ}{\mathbb{Z}}
\newcommand{\ant}{\mathrm{ant}}
\newcommand{\bant}{b\mathchar`-\mathrm{ant}}
\newcommand{\bsym}{b\mathchar`-\mathrm{sym}}
\newcommand{\tr}{\mathrm{tr}}
\newcommand{\wal}{\mathrm{wal}}
\newcommand{\wor}{\mathrm{wor}}
\newcommand{\Bcal}{\mathcal{B}}
\newcommand{\Ncal}{\mathcal{N}}
\newcommand{\Pcal}{\mathcal{P}}
\newcommand{\Qcal}{\mathcal{Q}}
\DeclareMathOperator*{\argmin}{arg\,min}

\allowdisplaybreaks

\begin{document}

\title{Quasi-Monte Carlo integration using digital nets with antithetics\thanks{This work was supported by JSPS Grant-in-Aid for Young Scientists No.15K20964.}}

\author{Takashi Goda\thanks{Graduate School of Engineering, The University of Tokyo, 7-3-1 Hongo, Bunkyo-ku, Tokyo 113-8656, Japan (\tt{goda@frcer.t.u-tokyo.ac.jp})}}

\date{\today}

\maketitle

\begin{abstract}
Antithetic sampling, which goes back to the classical work by Hammersley and Morton (1956), is one of the well-known variance reduction techniques for Monte Carlo integration. In this paper we investigate its application to digital nets over $\ZZ_b$ for quasi-Monte Carlo (QMC) integration, a deterministic counterpart of Monte Carlo, of functions defined over the $s$-dimensional unit cube. By looking at antithetic sampling as a geometric technique in a compact totally disconnected abelian group, we first generalize the notion of antithetic sampling from base $2$ to an arbitrary base $b\ge 2$. Then we analyze the QMC integration error of digital nets over $\ZZ_b$ with $b$-adic antithetics. Moreover, for a prime $b$, we prove the existence of good higher order polynomial lattice point sets with $b$-adic antithetics for QMC integration of smooth functions in weighted Sobolev spaces. Numerical experiments based on Sobol' point sets up to $s=100$ show that the rate of convergence can be improved for smooth integrands by using antithetic sampling technique, which is quite encouraging beyond the reach of our theoretical result and motivates future work to address.
\end{abstract}
\emph{Keywords}:\; Quasi-Monte Carlo, antithetic sampling, digital nets, higher order polynomial lattices, Walsh functions\\
\emph{MSC classifications}:\; 65C05, 65D30, 65D32.

%%%%%%%%%%%%%%%%%%%%%%%%%%%%%%%%%%%%%%%%%%%%%%%%%%%%%%%%%%%
%%%%%%%%%%%%%%%%%%%%%%%%%%%%%%%%%%%%%%%%%%%%%%%%%%%%%%%%%%%
%%%%%%%%%%%%%%%%%%%%%%%%%%%%%%%%%%%%%%%%%%%%%%%%%%%%%%%%%%%
\section{Introduction}\label{sec:intro}
In this paper we study multivariate integration of real-valued functions defined over the $s$-dimensional unit cube $[0,1]^s$. For a Riemann integrable function $f: [0,1]^s \to \RR$, we denote by $I(f)$ the true integral of $f$, i.e.,
\begin{align*}
I(f) := \int_{[0,1]^s}f(\bsx)\, \rd \bsx .
\end{align*}
As an approximate evaluation of $I(f)$, we consider 
\begin{align*}
I(f;P) := \frac{1}{N}\sum_{n=0}^{N-1}f(\bsx_n) ,
\end{align*}
where $P=\{\bsx_0,\bsx_1,\ldots,\bsx_{N-1}\}\subset [0,1]^s$ is a finite point set. Here points are counted according to their multiplicity.

If one chooses the points $\bsx_0,\bsx_1,\ldots,\bsx_{N-1}$ independently and randomly from $[0,1]^s$, the approximation $I(f;P)$ is called \emph{Monte Carlo (MC) integration}. The central limit theorem states that, for any function $f\in L^2([0,1]^s)$, the random variable $\sqrt{N}(I(f;P)-I(f))$ converges in distribution to a normal distribution $\Ncal(0,\sigma^2(f))$ as $N\to \infty$, where $\sigma^2(f)$ denotes the variance of $f$, i.e.,
\begin{align*}
\sigma^2(f) = \int_{[0,1]^s}\left(f(\bsx)-I(f)\right)^2\, \rd \bsx .
\end{align*}
Thus the MC integration has a probabilistic error of order $N^{-1/2}$. Here the rate of convergence is independent of $s$, although the variance of $f$ may depend on $s$. One of the most prominent ways to improve the MC integration error is to attempt reducing the variance of $f$, see for instance \cite[Chapter~4]{Lembook}.

Among many others, the method of \emph{antithetic variates}, also called \emph{antithetic sampling}, introduced by Hammersley and Morton \cite{HM56} is one of the simplest and best-known techniques for variance reduction. This method proceeds as follows: Let $\bsone$ denote the vector of $s$ 1's. For an even number $N$, let $\bsx_0,\bsx_1,\ldots,\bsx_{N/2-1}$ be chosen independently and randomly from $[0,1]^s$. For each point $\bsx_n$, we define $\tilde{\bsx}_n:=\bsone -\bsx_n$. Then the MC integration with antithetic variates is given by $I(f;P_{\ant})$ with 
\begin{align}\label{eq:antithetic_points}
P_{\ant}=\{\bsx_0,\bsx_1,\ldots,\bsx_{N/2-1},\tilde{\bsx}_0,\tilde{\bsx}_1,\ldots,\tilde{\bsx}_{N/2-1}\}.
\end{align}
The central limit theorem states that, again for any function $f\in L^2([0,1]^s)$, the random variable $\sqrt{N}(I(f;P_{\ant})-I(f))$ converges in distribution to a normal distribution with mean $0$ and variance 
\begin{align*}
\sigma^2(f) + \int_{[0,1]^s}\left(f(\bsx)-I(f)\right)\left(f(\bsone-\bsx)-I(f)\right)\, \rd \bsx .
\end{align*}
It is now obvious that the MC integration with antithetic variates is superior to the plain MC integration if the latter term in the last expression is negative, although the probabilistic error of order $N^{-1/2}$ remains unchanged.

\emph{Quasi-Monte Carlo (QMC) integration} aims at improving the rate of convergence by replacing random sample points with deterministically chosen points which are uniformly distributed in $[0,1]^s$. In the classical QMC theory, this replacement has been often motivated by the Koksma-Hlawka inequality, which states that, for any function $f$ with bounded variation in the sense of Hardy and Krause, we have
\begin{align*}
\left| I(f;P)-I(f)\right| \le V_{\mathrm{HK}}(f) D^{*}(P),
\end{align*}
where $V_{\mathrm{HK}}(f)$ denotes the total variation of $f$ in the sense of Hardy and Krause, and $D^{*}(P)$ the star-discrepancy of $P$, see for instance \cite[Chapter~3]{Nbook}. Thus in order to make the integration error small, it suffices to find a good point set whose star-discrepancy is small. In fact, there are several known explicit constructions of point sets whose star-discrepancy is of order $N^{-1+\varepsilon}$ with arbitrarily small $\varepsilon>0$, see for instance \cite{Faure82,Halton60,Nied88,NXbook,Sobol67}. Since the term $V_{\mathrm{HK}}(f)$ does not affect the rate of convergence, the QMC integration error decays much faster than the MC integration error.

Since a point set is taken deterministically for QMC integration and the variance of $f$ does not come into play in the error estimate, it is largely unknown whether variance reduction techniques can provide any benefit to QMC integration. As far as the author knows, there are only a handful of papers on application of variance reduction techniques to QMC integration. These include importance sampling \cite{AD15,Chelson76,SM94}, control variates \cite{HLO05}, and a variant of antithetic sampling (named local antithetic sampling) \cite{Owen08}. Note that the last two cited papers deal with, instead of deterministic QMC integration, \emph{randomized QMC (RQMC) integration} which applies a randomizing transformation to point sets such that their essential equi-distribution property is preserved. Therefore, properly speaking, a point set is not taken completely deterministically therein.

In this paper we investigate a combination of deterministic QMC integration with antithetic sampling. We consider a special class of point sets called \emph{digital nets over $\ZZ_b$} for an integer base $b\ge 2$. Although digital nets are usually defined by using generating matrices where each column consists of only finitely many non-zero entries, such a definition does not suffice for our error analysis. This means that we have to permit infinite-column generating matrices, i.e., generating matrices whose each column can contain infinitely many non-zero entries. In fact, this issue has been recently discussed in \cite{GSYxx}.

By looking at antithetic sampling as a geometric technique in a compact totally disconnected abelian group, the original antithetic sampling as in (\ref{eq:antithetic_points}) can be combined quite well with digital nets over $\ZZ_2$ but not so much with digital nets over $\ZZ_b$ for $b\ge 3$. Based on an idea similar to that of \cite{GSY15,GSYxx} as well as \cite{Godaxx}, in which the notions of tent transformation and symmetrization are generalized from base $2$ to an arbitrary base $b\ge 2$, respectively, we first generalize the notion of antithetic sampling from base $2$ to an arbitrary base $b\ge 2$ in this paper. Then we analyze the QMC integration error of digital nets over $\ZZ_b$ with $b$-adic antithetics. This shall be done in Section~\ref{sec:banti}, which is the first contribution of this paper.

Using the result of Section~\ref{sec:banti}, we give one example of how the use of $b$-adic antithetics brings a noticeable benefit to QMC integration. In particular, we prove the existence of higher order polynomial lattice point sets with $b$-adic antithetics which achieve almost the optimal rate of convergence for smooth functions in weighted Sobolev spaces, among a smaller number of candidates as compared to that of \cite{DP07}. This shall be done in Section~\ref{sec:exist}, which is the second contribution of this paper. Hence it would be interesting to study how to find such good point sets in a constructive manner, which we leave open for future work to address.

Finally in Section~\ref{sec:numer}, we conduct some numerical experiments up to $s=100$ based on Sobol' point sets, which are a special construction of digital nets over $\ZZ_2$. For smooth test integrands, we compare the performances of Sobol' point sets with and without dyadic antithetics. Surprisingly, it turns out that the rate of error convergence is improved by the use of antithetics. At this moment, however, there is no theoretical foundation to comprehend this nice convergence behavior. Hence, our numerical results motivate further work on a combination of QMC integration with antithetic sampling, and more broadly, with variance reduction techniques.
%%%%%%%%%%%%%%%%%%%%%%%%%%%%%%%%%%%%%%%%%%%%%%%%%%%%%%%%%%%
%%%%%%%%%%%%%%%%%%%%%%%%%%%%%%%%%%%%%%%%%%%%%%%%%%%%%%%%%%%
%%%%%%%%%%%%%%%%%%%%%%%%%%%%%%%%%%%%%%%%%%%%%%%%%%%%%%%%%%%
\section{Preliminaries}\label{sec:pre}
We shall use the following notation throughout this paper. Let $\NN$ be the set of positive integers and $\NN_0:=\NN \cup \{0\}$. Let $\CC$ be the set of all complex numbers. For an integer $b\ge 2$, let $\ZZ_b$ be the residue class ring modulo $b$, which we identify with the set $\{0,1,\ldots,b-1\}$ equipped with addition and subtraction modulo $b$, denoted by $\oplus$ and $\ominus$, respectively. For any point $x\in [0,1]$, we always use the $b$-adic expansion $x=\xi_1/b+\xi_2/b^2+\cdots$ with $\xi_i\in \ZZ_b$, which is unique in the sense that infinitely many of the $\xi_i$'s are different from $b-1$ if $x\in [0,1)$ and that all the $\xi_i$'s are equal to $b-1$ if $x=1$. Note that for $1\in \NN$ we use the $b$-adic expansion $1\cdot b^0$, whereas for $1\in [0,1]$ we use the $b$-adic expansion $(b-1)b^{-1}+(b-1)b^{-2}+\cdots$. It will be always clear from the context which expansion we use.

In this section, we recall necessary background and further notation, which shall be used in the subsequent analysis.
%%%%%%%%%%%%%%%%%%%%%%%%%%%%%%%%%%%%%%%%%%%%%%%%%%%%%%%%%%%
\subsection{Walsh functions}\label{subsec:walsh}
Walsh functions play an important role in analyzing the QMC integration error when using digital nets. We refer to \cite[Appendix~A]{DPbook} for general information on Walsh functions in the context of QMC integration. We first define Walsh functions for the one-dimensional case. In the following, let $\omega_b$ denote the primitive root of unity $\exp(2\pi \sqrt{-1}/b)$.

\begin{definition}
Let $k\in \NN_0$ with its $b$-adic expansion $k=\kappa_0+\kappa_1 b+\cdots$, which is actually a finite expansion. Then the $k$-th $b$-adic Walsh function $\wal_k:[0,1]\to \{1,\omega_b,\ldots,\omega_b^{b-1}\}$ is defined by
\begin{align*}
 \wal_k(x) := \omega_b^{\kappa_0\xi_1+\kappa_1\xi_2+\cdots} ,
\end{align*}
for $x\in [0,1]$ with its unique $b$-adic expansion $x=\xi_1/b+\xi_2/b^2+\cdots$.
\end{definition}
\noindent
Suppose that the $b$-adic expansion of $k$ is given by $k=\kappa_0+\kappa_1 b+\cdots+\kappa_{a-1}b^{a-1}$ with $\kappa_{a-1}\ne 0$. Then it is obvious from the above definition that the function $\wal_k$ does not depend on the digits $\xi_{a+1},\xi_{a+2},\ldots$, which appear in the $b$-adic expansion of $x$. This implies that every Walsh function $\wal_k$ is a piecewise constant function.

We can generalize the definition of Walsh functions for the high-dimensional case as follows.
\begin{definition}
Let $\bsk=(k_1,\ldots,k_s)\in \NN_0^s$. Then the $\bsk$-th $b$-adic Walsh function $\wal_{\bsk}:[0,1]^s\to \{1,\omega_b,\ldots,\omega_b^{b-1}\}$ is defined by
\begin{align*}
 \wal_{\bsk}(\bsx) := \prod_{j=1}^{s}\wal_{k_j}(x_j) ,
\end{align*}
for $\bsx=(x_1,\ldots,x_s)\in [0,1]^s$.
\end{definition}

It is known that, for fixed $b,s\in \NN$, $b\ge 2$, the $b$-adic Walsh function system $\{\wal_{\bsk}\colon \bsk\in \NN_0^s\}$ is a complete orthonormal basis in $L^2([0,1]^s)$, see for instance \cite[Theorem~A.11]{DPbook}. Therefore, every function $f\in L^2([0,1]^s)$ has its Walsh series expansion
\begin{align*}
 \sum_{\bsk\in \NN_0^s}\hat{f}(\bsk)\wal_{\bsk}, 
\end{align*}
where $\hat{f}(\bsk)$ denotes the $\bsk$-th Walsh coefficient which is defined by
\begin{align*}
 \hat{f}(\bsk) := \int_{[0,1]^s}f(\bsx)\overline{\wal_{\bsk}(\bsx)}\, \rd \bsx .
\end{align*}
Moreover, let $f:[0,1]^s\to \RR$ be a continuous function which satisfies the condition $\sum_{\bsk\in \NN_0^s}|\hat{f}(\bsk)|<\infty$. Then the Walsh series expansion of $f$ converges to $f$ itself pointwise absolutely, i.e., for any $\bsx\in [0,1]^s$, we have
\begin{align*}
 f(\bsx) = \sum_{\bsk\in \NN_0^s}\hat{f}(\bsk)\wal_{\bsk}(\bsx),
\end{align*}
see for instance \cite[Appendix~A.3]{DPbook} and \cite[Lemma~18]{GSY15}.

%%%%%%%%%%%%%%%%%%%%%%%%%%%%%%%%%%%%%%%%%%%%%%%%%%%%%%%%%%%
\subsection{Infinite direct products of $\ZZ_b$}\label{subsec:inf_prod}
In order to permit digital nets over $\ZZ_b$ which are defined by using infinite-column generating matrices, as mentioned in Section~\ref{sec:intro}, we have to deal with the infinite direct product of $\ZZ_b$, which is denoted by $G:=\prod_{i\ge 1}\ZZ_b$. Here we essentially follow the exposition of \cite[Subsection~2.1]{GSYxx}.

$G$ is a compact totally disconnected abelian group with the product topology, where $\ZZ_b$ is considered to be a discrete group. With a slight abuse of notation we denote by $\oplus$ and $\ominus$ addition and subtraction in $G$, respectively. Let $\tilde{\mu}$ be the product measure on $G$ induced by the equi-probability measure on $\ZZ_b$. A character on $G$ is a continuous group homomorphism from $G$ to $\{z\in \CC\colon |z|=1\}$. For $k\in \NN_0$, the $k$-th character is defined as follows.

\begin{definition}
Let $k\in \NN_0$ with its $b$-adic expansion $k=\kappa_0+\kappa_1 b+\cdots$, which is actually a finite expansion. Then the $k$-th character $\chi_k:G\to \{1,\omega_b,\ldots,\omega_b^{b-1}\}$ is defined by
\begin{align*}
 \chi_k(z) := \omega_b^{\kappa_0\zeta_1+\kappa_1\zeta_2+\cdots} ,
\end{align*}
for $z=(\zeta_1,\zeta_2,\ldots)\in G$.
\end{definition}
\noindent Note that every character on $G$ is equal to some $\chi_k$, see \cite{Pbook}.

The group $G$ can be related to the unit interval $[0,1]$ as follows: Let $z=(\zeta_1,\zeta_2,\ldots)\in G$ and $x\in [0,1]$ with its unique $b$-adic expansion $x=\xi_1/b+\xi_2/b^2+\cdots$. Then the \emph{projection map} $\pi:G\to [0,1]$ is defined by
\begin{align*}
 \pi(z) := \zeta_1/b+\zeta_2/b^2+\cdots ,
\end{align*}
whereas the \emph{section map} $\sigma:[0,1]\to G$ is defined by
\begin{align*}
 \sigma(x) := (\xi_1,\xi_2,\ldots) .
\end{align*}
By definition, $\pi$ is surjective and $\sigma$ is injective. In addition, we note that $\pi$ is continuous and that $\pi \circ \sigma = \mathrm{id}_{[0,1]}$. 

Now let us consider the $s$-ary Cartesian product of $G$, denoted by $G^s$. Again $G^s$ is a compact totally disconnected abelian group with the product topology. The operators $\oplus$ and $\ominus$ are applied componentwise. Moreover, let $\tilde{\bsmu}$ be the product measure on $G^s$ induced by $\tilde{\mu}$. For $\bsk\in \NN_0^s$, the $\bsk$-th character is defined as follows.
\begin{definition}
Let $\bsk=(k_1,\ldots,k_s)\in \NN_0^s$. Then the $\bsk$-th character $\chi_{\bsk}:G^s\to \{1,\omega_b,\ldots,\omega_b^{b-1}\}$ is defined by
\begin{align*}
 \chi_{\bsk}(\bsz) := \prod_{j=1}^{s}\chi_{k_j}(z_j) ,
\end{align*}
for $\bsz=(z_1,\ldots,z_s)\in G^s$.
\end{definition}
\noindent Again note that every character on $G^s$ is equal to some $\chi_{\bsk}$. The group $G^s$ can be related to the unit cube $[0,1]^s$ by applying both $\pi$ and $\sigma$ componentwise. Some important facts are summarized below. We refer to \cite{Pbook} and \cite{SWSbook} for the proofs of the first two items and the remaining three items, respectively. Although the reference \cite{SWSbook} only deals with the dyadic ($b=2$) case, the proofs for an arbitrary integer $b\ge 2$ remain essentially the same.

\begin{proposition}\label{prop:inf_prod} The following holds true:
\begin{enumerate}
\item For $k \in \NN_0$, we have
  \begin{align*}
    \int_{G} \chi_{k}(z)\, \rd \tilde{\mu}(z) = \begin{cases}
     1 & \text{if $k=0$},  \\
     0 & \text{otherwise}.
    \end{cases}
  \end{align*}
\item For $\bsk,\bsl\in \NN_0^s$, we have
  \begin{align*}
    \int_{G^s} \chi_{\bsk}(\bsz)\overline{\chi_{\bsl}(\bsz)}\, \rd \tilde{\bsmu}(\bsz) = \begin{cases}
     1 & \text{if $\bsk=\bsl$},  \\
     0 & \text{otherwise}.
    \end{cases}
  \end{align*}
\item For any $f \in L^1([0,1]^s)$, we have
  \begin{align*}
    \int_{[0,1]^s} f(\bsx) \, \rd \bsx = \int_{G^s} f(\pi(\bsz))\, \rd \tilde{\bsmu}(\bsz).
  \end{align*}
\item For any $g \in L^1(G^s)$, we have
  \begin{align*}
    \int_{G^s} g(\bsz)\, \rd \tilde{\bsmu}(\bsz) = \int_{[0,1]^s} g (\sigma(\bsx))\, \rd \bsx.
  \end{align*}
\item Let $H_n:=\{z=(\zeta_1, \zeta_2, \dots) \in G: \zeta_1=\zeta_2=\cdots =\zeta_n=0 \}$. Then we have
  \begin{align*}
    \sum_{\substack{\bsk \in \NN_0^s\\ k_j<b^n, \forall j}} \chi_{\bsk}(\bsz) = \begin{cases}
     b^{sn} & \text{if $\bsz \in H_n^s$}, \\
     0 & \text{otherwise}.
\end{cases}
  \end{align*}
\end{enumerate}
\end{proposition}

%%%%%%%%%%%%%%%%%%%%%%%%%%%%%%%%%%%%%%%%%%%%%%%%%%%%%%%%%%%
\subsection{Digital nets over $\ZZ_b$}\label{subsec:digital_net}
We now introduce the definition of digital nets over $\ZZ_b$ by using infinite-column generating matrices.

\begin{definition}\label{def:digital_net}
For $m,s\in \NN$, let $C_1,\ldots,C_s\in \ZZ_b^{\NN\times m}$. Let $h$ be an integer with $0\le h<b^m$ whose $b$-adic expansion is denoted by $h=\eta_0+\eta_1b+\ldots+\eta_{m-1}b^{m-1}$. Let $\bsz_h=(z_{h,1},\ldots,z_{h,s})\in G^s$ be given by
\begin{align*}
z_{h,j}^{\top} = C_j\cdot (\eta_0,\eta_1,\ldots,\eta_{m-1})^{\top}\quad \text{for $1\le j\le s$}.
\end{align*}
Then the set $\Pcal=\{\bsz_0,\bsz_1,\ldots,\bsz_{b^m-1}\}\subset G^s$ is called a digital net over $\ZZ_b$ in $G^s$ with generating matrices $C_1,\ldots,C_s$.

Furthermore, the set $P:=\{\pi(\bsz)\colon \bsz\in \Pcal\}\subset [0,1]^s$ is called a digital net over $\ZZ_b$ in $[0,1]^s$ with generating matrices $C_1,\ldots,C_s$.
\end{definition}

In the remainder of this paper, digital nets in $G^s$ are denoted by the calligraphic letter $\Pcal$, whereas digital nets in $[0,1]^s$ are denoted by the block letter $P$, as in the above definition. Since $P$ is nothing but the image of $\Pcal$ under $\pi:G^s\to [0,1]^s$, we shall mostly deal with $\Pcal$ instead of $P$ and often write $\pi(\Pcal)$ instead of $P$ to represent digital nets in $[0,1]^s$. Note that every digital net in $G^s$ is a $\ZZ_b$-module of $G^s$ as well as a subgroup of $G^s$.

For a digital net $\Pcal$ in $G^s$, its dual net is defined as follows.
\begin{definition}\label{def:dual_net}
For $m,s\in \NN$, let $\Pcal$ be a digital net in $G^s$ with generating matrices $C_1,\ldots,C_s\in \ZZ_b^{\NN\times m}$. Then the dual net of $\Pcal$, denoted by $\Pcal^{\perp}$, is defined by
\begin{align*}
\Pcal^{\perp} := \left\{ \bsk=(k_1,\ldots,k_s)\in \NN_0^s\colon \vec{k}_1 C_1\oplus \cdots \oplus \vec{k}_sC_s=(0,\ldots,0) \in \ZZ_b^m \right\} ,
\end{align*}
where we write $\vec{k}=(\kappa_0,\kappa_1,\ldots)$ for $k\in \NN_0$ with its finite $b$-adic expansion $k=\kappa_0+\kappa_1b+\cdots$.
\end{definition}

We recall that the set of $\chi_{\bsk}$'s are the characters on $G^s$. From the group structure of $\Pcal$ and Definition~\ref{def:dual_net}, we have the following lemma.
\begin{lemma}\label{lem:dual_char}
Let $\Pcal$ be a digital net in $G^s$ and $\Pcal^{\perp}$ its dual net. Then we have
\begin{align*}
\sum_{\bsz\in \Pcal}\chi_{\bsk}(\bsz) = \begin{cases}
|\Pcal| & \text{if $\bsk\in \Pcal^{\perp}$,} \\
0 & \text{otherwise.}
\end{cases}
\end{align*}
\end{lemma}

%%%%%%%%%%%%%%%%%%%%%%%%%%%%%%%%%%%%%%%%%%%%%%%%%%%%%%%%%%%
%%%%%%%%%%%%%%%%%%%%%%%%%%%%%%%%%%%%%%%%%%%%%%%%%%%%%%%%%%%
%%%%%%%%%%%%%%%%%%%%%%%%%%%%%%%%%%%%%%%%%%%%%%%%%%%%%%%%%%%
\section{Digital nets with antithetics}\label{sec:banti}
In this section, we generalize the notion of antithetic sampling from base $2$ to an arbitrary base $b\ge 2$, and then analyze the QMC integration error of digital nets over $\ZZ_b$ with $b$-adic antithetics.

%%%%%%%%%%%%%%%%%%%%%%%%%%%%%%%%%%%%%%%%%%%%%%%%%%%%%%%%%%%
\subsection{Generalization of antithetic sampling}
In order to give a hint as to how we generalize the notion of antithetic sampling, we first give another look at the original antithetic sampling.

Here let us consider the dyadic ($b=2$) case. Let $e:=(1,1,\ldots)\in G$. Then it obviously holds that $\pi(e)=1$. For any $z=(\zeta_1,\zeta_2,\ldots)\in G$, we have
\begin{align*}
1-\pi(z) & = \pi(e)-\pi(z) \\
& = \left(\frac{1}{2}+\frac{1}{2^2}+\cdots \right) - \left(\frac{\zeta_1}{2}+\frac{\zeta_2}{2^2}+\cdots \right) \\
& = \frac{1-\zeta_1}{2}+\frac{1-\zeta_2}{2^2}+\cdots \\
& = \pi(\zeta_1\oplus 1, \zeta_2\oplus 1, \ldots) = \pi(z\oplus e).
\end{align*}
This means that $\pi(z\oplus e)$ is the antithetic of $\pi(z)$. In this interpretation, the antithetic of $1/2$ should be understood as $1/2^2+1/2^3+\cdots$ not as $1/2$, although the expansion $1/2^2+1/2^3+\cdots$ is not allowed due to the uniqueness of dyadic expansion for $x\in [0,1]$. The same problem arises whenever $x$ is a dyadic rational, i.e., $x$ is given in the form $a/2^c$ with $a,c\in \NN_0$ and $0\le a\le 2^c$. This is why we consider the infinite direct product of $\ZZ_2$, which permits different dyadic expansions for $x\in [0,1]$ through the projection map $\pi$. For instance, we have $\pi(1,0,0,\ldots)=\pi(0,1,1,\ldots)=1/2$.

For the $s$-dimensional case, let $\bse:=(e,\ldots,e)\in G^s$. Then for any $\bsz=(z_1,\ldots,z_s)\in G^s$ we have
\begin{align*}
\bsone -\pi(\bsz) = \pi(\bse)-\pi(\bsz) = (\pi(z_1\oplus e),\ldots,\pi(z_s\oplus e)) = \pi(\bsz\oplus \bse).
\end{align*}
From the above identity, the original (dyadic) antithetic sampling can be seen as follows: Let $\Pcal$ be a finite set in $G^s$ and $P=\{\pi(\bsz)\colon \bsz\in \Pcal\}\in [0,1]^s$. Then $P_{\ant}$ is given by
\begin{align*}
P_{\ant} = P\cup \{\pi(\bsz\oplus \bse)\colon \bsz\in \Pcal \}.
\end{align*}

Now we are ready to introduce the notion of $b$-adic antithetic sampling. In the following, let $b$ be an arbitrary integer base $b\ge 2$. For $l\in \ZZ_b$, we write $\bse_l=(e_l,\ldots,e_l)\in G^s$ where $e_l$ is defined by $e_l:=(l,l,\ldots)\in G$.
\begin{definition}\label{def:badic_antithetic}
Let $\Pcal$ be a finite set in $G^s$. The $b$-adic antithetic sampling of $\Pcal$ is defined by
\begin{align*}
\Pcal_{\bant} := \bigcup_{l\in \ZZ_b}\{\bsz\oplus \bse_l\colon \bsz\in \Pcal \} .
\end{align*}
Furthermore, let $P=\{\pi(\bsz)\colon \bsz\in \Pcal\}$ be a finite point set in $[0,1]^s$. The $b$-adic antithetic sampling of $P$ is defined by $P_{\bant} := \{\pi(\bsz)\colon \bsz\in \Pcal_{\bant}\}$.
\end{definition}
\noindent
By definition, we have $|\Pcal_{\bant}|=b|\Pcal|$ and $|P_{\bant}|=b|P|$.

\begin{remark}
For $\bsl=(l_1,\ldots,l_s)\in \ZZ_b^s$, let $\bse_{\bsl}=(e_{l_1},\ldots,e_{l_s})\in G^s$. For a finite set $\Pcal\subset G^s$, the $b$-adic symmetrization of $\Pcal$ introduced in \cite{Godaxx} is defined by
\begin{align*}
 \Pcal_{\bsym} := \bigcup_{\bsl\in \ZZ_b^s}\{\bsz\oplus \bse_{\bsl}\colon \bsz\in \Pcal \} .
\end{align*}
Obviously we have $|\Pcal_{\bsym}|=b^s|\Pcal|$, so that the number of points grows exponentially with the dimension $s$. The $b$-adic antithetic sampling avoids such an exponential growth by considering only the case $l_1=\cdots=l_s$.
\end{remark}

%%%%%%%%%%%%%%%%%%%%%%%%%%%%%%%%%%%%%%%%%%%%%%%%%%%%%%%%%%%
\subsection{Digital nets with antithetics}
In this subsection and in the remainder of this paper, we focus on the case where the set $\Pcal$ (the point set $P$) is a digital net over $\ZZ_b$ in $G^s$ (in $[0,1]^s$, respectively).
\begin{lemma}\label{lem:bant-digital-net}
Let $\Pcal$ be a digital net over $\ZZ_b$ in $G^s$ with generating matrices $C_1,\ldots,C_s\in \ZZ_b^{\NN\times m}$. Then $\Pcal_{\bant}$ is a digital net over $\ZZ_b$ in $G^s$ with generating matrices $D_1,\ldots,D_s\in \ZZ_b^{\NN\times (m+1)}$, where each $D_j$ is given by $$D_j=(C_j\vert (1,1,\ldots)^{\top}).$$
\end{lemma}

\begin{proof}
Let $\Qcal$ denote the digital net over $\ZZ_b$ in $G^s$ with generating matrices $D_1,\ldots,D_s\in \ZZ_b^{\NN\times (m+1)}$. Then it suffices to prove $\Qcal=\Pcal_{\bant}$.

Let $\Pcal=\{\bsz_0,\ldots,\bsz_{b^m-1}\}$ and $\Qcal=\{\bsw_0,\ldots,\bsw_{b^{m+1}-1}\}$, where each element is given as in Definition~\ref{def:digital_net}. Now let $h$ be an integer with $0\le h<b^{m+1}$. We write $h=h'+lb^m$ with $h',l\in \NN_0$, $0\le h'<b^m$ and $0\le l<b$. Moreover, we denote the $b$-adic expansion of $h'$ by $h'=\eta_0+\eta_1b+\ldots+\eta_{m-1}b^{m-1}$. Then the $h$-th element $\bsw_h=(w_{h,1},\ldots,w_{h,s})$ of $\Qcal$ is given by
\begin{align*}
w_{h,j}^{\top} & = D_j\cdot (\eta_0,\eta_1,\ldots,\eta_{m-1},l)^{\top} \\
& = C_j\cdot (\eta_0,\eta_1,\ldots,\eta_{m-1})^{\top} \oplus (l,l,\ldots)^{\top} \\
& = z_{h',j}^{\top}\oplus e_l^{\top} ,
\end{align*}
from which it holds that $\bsw_h=\bsz_{h'}\oplus \bse_l$. Thus we have
\begin{align*}
\Qcal & = \{\bsw_{h'+lb^m} \colon 0\le h'<b^m, 0\le l<b\} \\
& = \{\bsz_{h'}\oplus \bse_l \colon 0\le h'<b^m, 0\le l<b\} \\
& = \bigcup_{l\in \ZZ_b}\{\bsz\oplus \bse_l\colon \bsz\in \Pcal \} = \Pcal_{\bant},
\end{align*}
which completes the proof.
\end{proof}
\noindent
From this lemma, it is obvious that $P_{\bant}$ is a digital net over $\ZZ_b$ in $[0,1]^s$ with generating matrices $D_1,\ldots,D_s\in \ZZ_b^{\NN\times (m+1)}$.

In the remainder of this paper, we need the \emph{sum-of-digit modulo $b$} function $\delta: \NN_0\to \{0,1,\ldots,b-1\}$, which is defined as follows. For $k\in \NN_0$, we denote its $b$-adic expansion by $k=\kappa_{0}+\kappa_{1}b+\cdots$, which is actually a finite expansion. Then we define
\begin{align*}
\delta(k) := \sum_{i\ge 0}\kappa_i \pmod b.
\end{align*}
For $\bsk=(k_1,\ldots,k_s)\in \NN_0^s$, we define
\begin{align*}
\delta(\bsk) := \sum_{j=1}^{s}\delta(k_j) \pmod b.
\end{align*}

The dual net of $\Pcal_{\bant}$ can be related to the dual net of $\Pcal$ as follows.
\begin{lemma}\label{lem:bani-dual-net}
Let $\Pcal$ be a digital net over $\ZZ_b$ in $G^s$ and $\Pcal^{\perp}$ its dual net. Then the dual net of $\Pcal_{\bant}$ is given by
\begin{align*}
\Pcal_{\bant}^{\perp} = \Pcal^{\perp}\cap \left\{\bsk\in \NN_0^s: \delta(\bsk)= 0\right\}.
\end{align*}
\end{lemma}

\begin{proof}
From Definition~\ref{def:dual_net} and Lemma~\ref{lem:bant-digital-net}, the dual net of $\Pcal_{\bant}$ is given by
\begin{align*}
\Pcal_{\bant}^{\perp} := \left\{ \bsk=(k_1,\ldots,k_s)\in \NN_0^s\colon \vec{k}_1 D_1\oplus \cdots \oplus \vec{k}_sD_s=(0,\ldots,0) \in \ZZ_b^{m+1} \right\} .
\end{align*}
In the above, we have
\begin{align*}
\vec{k}_1 D_1\oplus \cdots \oplus \vec{k}_sD_s & =  \vec{k}_1 (C_1\vert (1,1,\ldots)^{\top})\oplus \cdots \oplus \vec{k}_s (C_s\vert (1,1,\ldots)^{\top}) \\
& = (\vec{k}_1 C_1\vert \delta(k_1))\oplus \cdots \oplus (\vec{k}_s C_s\vert \delta(k_s)) \\
& = (\vec{k}_1 C_1\oplus \cdots \oplus \vec{k}_s C_s \vert \delta(\bsk)).
\end{align*}
Thus the condition $\bsk\in \Pcal_{\bant}^{\perp}$ is satisfied if and only if 
\begin{align*}
\vec{k}_1 C_1\oplus \cdots \oplus \vec{k}_s C_s = (0,\ldots,0) \in \ZZ_b^m\quad \text{and}\quad \delta(\bsk)= 0 ,
\end{align*}
which proves this lemma.
\end{proof}

%%%%%%%%%%%%%%%%%%%%%%%%%%%%%%%%%%%%%%%%%%%%%%%%%%%%%%%%%%%
\subsection{QMC integration error}
Here we investigate the QMC integration error of digital nets over $\ZZ_b$ with $b$-adic antithetics. First we study the integration error for a particular function, and then study the worst-case error in a reproducing kernel Hilbert space.

In order to study the integration error for a particular function $f$, we need the following lemma on the pointwise absolute convergence of the Walsh series. Although the proof is quite similar to that used in \cite[Proposition~19]{GSYxx}, we provide it below for the sake of completeness.
\begin{lemma}\label{lem:abs_conv_f}
Let $f:[0,1]^s\to \RR$ be a continuous function which satisfies the condition $\sum_{\bsk\in \NN_0^s}|\hat{f}(\bsk)|<\infty$. Then for any $\bsz\in G^s$ we have
\begin{align}\label{eq:abs_conv_f}
f(\pi(\bsz)) = \sum_{\bsk\in \NN_0^s}\hat{f}(\bsk)\chi_{\bsk}(\bsz).
\end{align}
\end{lemma}

\begin{proof}
Due to the condition $\sum_{\bsk\in \NN_0^s}|\hat{f}(\bsk)|<\infty$, the right-hand side of (\ref{eq:abs_conv_f}) converges absolutely. Thus it suffices to prove
\begin{align*}
\lim_{n\to \infty}\sum_{\substack{\bsk\in \NN_0^s\\ k_j<b^n, \forall j}}\hat{f}(\bsk)\chi_{\bsk}(\bsz) = f(\pi(\bsz)).
\end{align*}
Since $\pi \circ \sigma = \mathrm{id}_{[0,1]^s}$ and $\chi_{\bsk}(\sigma(\bsx))=\wal_{\bsk}(\bsx)$ for any $\bsk\in \NN_0^s$ and $\bsx\in [0,1]^s$, the sum on the left-hand side above can be rewritten as
\begin{align*}
\sum_{\substack{\bsk\in \NN_0^s\\ k_j<b^n, \forall j}}\hat{f}(\bsk)\chi_{\bsk}(\bsz) & = \sum_{\substack{\bsk\in \NN_0^s\\ k_j<b^n, \forall j}}\chi_{\bsk}(\bsz) \int_{[0,1]^s}f(\bsx)\overline{\wal_{\bsk}(\bsx)}\, \rd \bsx \\
& = \sum_{\substack{\bsk\in \NN_0^s\\ k_j<b^n, \forall j}}\chi_{\bsk}(\bsz) \int_{[0,1]^s}(f\circ \pi \circ \sigma)(\bsx)\overline{\chi_{\bsk}(\sigma(\bsx))}\, \rd \bsx \\
& = \sum_{\substack{\bsk\in \NN_0^s\\ k_j<b^n, \forall j}}\chi_{\bsk}(\bsz) \int_{G^s}(f\circ \pi)(\bsw)\overline{\chi_{\bsk}(\bsw)}\, \rd \tilde{\bsmu}(\bsw) \\
& = \int_{G^s}f(\pi(\bsw))\sum_{\substack{\bsk\in \NN_0^s\\ k_j<b^n, \forall j}}\chi_{\bsk}(\bsz\ominus \bsw)\, \rd \tilde{\bsmu}(\bsw) ,
\end{align*}
where we used Item~4 of Proposition~\ref{prop:inf_prod} in the third equality. Let us define the set $H(\bsz,n)=\{\bsw\in G^s\colon \bsz\ominus \bsw \in H_n^s\}$, where $H_n$ is defined as in Item~5 of Proposition~\ref{prop:inf_prod}. Then for any $\bsz\in G^s$ it holds that $\tilde{\bsmu}(H(\bsz,n))=b^{-ns}$ and 
  \begin{align*}
    \sum_{\substack{\bsk\in \NN_0^s\\ k_j<b^n, \forall j}}\chi_{\bsk}(\bsz\ominus \bsw)= \begin{cases}
     b^{sn} & \text{if $\bsw \in H(\bsz,n)$}, \\
     0 & \text{otherwise}.
\end{cases}
  \end{align*}
Therefore, we have
\begin{align*}
\sum_{\substack{\bsk\in \NN_0^s\\ k_j<b^n, \forall j}}\hat{f}(\bsk)\chi_{\bsk}(\bsz) & = b^{ns}\int_{H(\bsz,n)}f(\pi(\bsw))\, \rd \tilde{\bsmu}(\bsw) \\
& = \frac{1}{\tilde{\bsmu}(H(\bsz,n))}\int_{H(\bsz,n)}f(\pi(\bsw))\, \rd \tilde{\bsmu}(\bsw) \\
& \to f(\pi(\bsz)) \quad \text{as $n\to \infty$},
\end{align*}
where we have the last convergence since $f\circ \pi$ is continuous from the fact that both $f$ and $\sigma$ are continuous.
\end{proof}

For a particular function $f$ which satisfies the continuity and summability conditions in the above lemma, the signed QMC integration error of digital nets over $\ZZ_b$ can be given as follows.
\begin{lemma}
Let $\Pcal$ be a digital net over $\ZZ_b$ in $G^s$ and $\Pcal^{\perp}$ its dual net. For any continuous function $f:[0,1]^s\to \RR$ which satisfies the condition $\sum_{\bsk\in \NN_0^s}|\hat{f}(\bsk)|<\infty$, we have
\begin{align*}
I(f;\pi(\Pcal)) - I(f) = \sum_{\bsk\in \Pcal^{\perp}\setminus \{\bszero\}}\hat{f}(\bsk).
\end{align*}
\end{lemma}

\begin{proof}
By the definition of Walsh functions, it holds that $I(f)=\hat{f}(\bszero)$. Using the results of Lemmas~\ref{lem:abs_conv_f} and \ref{lem:dual_char}, we have
\begin{align*}
I(f;\pi(\Pcal)) - I(f) & = \frac{1}{|\Pcal|}\sum_{\bsz\in \Pcal}f(\pi(\bsz))-\hat{f}(\bszero) \\
& = \frac{1}{|\Pcal|}\sum_{\bsz\in \Pcal}\sum_{\bsk\in \NN_0^s}\hat{f}(\bsk)\chi_{\bsk}(\bsz) -\hat{f}(\bszero) \\
& = \sum_{\bsk\in \NN_0^s}\hat{f}(\bsk)\frac{1}{|\Pcal|}\sum_{\bsz\in \Pcal}\chi_{\bsk}(\bsz) -\hat{f}(\bszero) \\
& = \sum_{\bsk\in \Pcal^{\perp}}\hat{f}(\bsk)-\hat{f}(\bszero) = \sum_{\bsk\in \Pcal^{\perp}\setminus \{\bszero\}}\hat{f}(\bsk).
\end{align*}
\end{proof}

Combining the above result with Lemma~\ref{lem:bani-dual-net}, we have the following.
\begin{theorem}\label{thm:QMC_error_f}
Let $\Pcal$ be a digital net over $\ZZ_b$ in $G^s$ and $\Pcal^{\perp}$ its dual net. For any continuous function $f:[0,1]^s\to \RR$ which satisfies the condition $\sum_{\bsk\in \NN_0^s}|\hat{f}(\bsk)|<\infty$, we have
\begin{align*}
I(f;\pi(\Pcal_{\bant}))-I(f)= \sum_{\substack{ \bsk\in \Pcal^{\perp}\setminus \{\bszero\}\\ \delta(\bsk)=0}}\hat{f}(\bsk) .
\end{align*}
\end{theorem}

\begin{remark}\label{rem:QMC_error_f}
In general, we cannot expect a cancellation of $\hat{f}(\bsk)$. Thus, it is often the case that the absolute integration error is considered instead of the signed integration error. In this case, due to the triangle inequality, we have the following error bound
\begin{align*}
\lvert I(f;\pi(\Pcal_{\bant}))-I(f)\rvert \le \sum_{\substack{ \bsk\in \Pcal^{\perp}\setminus \{\bszero\}\\ \delta(\bsk)=0}}|\hat{f}(\bsk)|.
\end{align*}
The right-hand side above is always less than or equal to $\sum_{\bsk\in \Pcal^{\perp}\setminus \{\bszero\}}|\hat{f}(\bsk)|$, which is a bound on $\lvert I(f;\pi(\Pcal))-I(f)\rvert$.
\end{remark}

Let us move on to the worst-case error in a reproducing kernel Hilbert space (RHKS). Let $H$ be a RHKS with reproducing kernel $K:[0,1]^s\times [0,1]^s \to \RR$. We denote the inner product in $H$ by $\langle f,g \rangle_{H}$ for $f,g\in H$ and its associated norm by $\lVert f\rVert_{H}:=\sqrt{\langle f,f \rangle_{H}}$. The worst-case error in $H$ of QMC integration using a finite point set $P\subset [0,1]^s$ is defined by
\begin{align*}
e^{\wor}(H;P) := \sup_{\substack{f\in H \\ \lVert f\rVert_{H}\le 1}}\left| I(f;P)-I(f)\right|.
\end{align*}
It is known that if a reproducing kernel $K$ satisfies $\int_{[0,1]^s}\sqrt{K(\bsx,\bsx)}\, \rd \bsx< \infty$, we have
\begin{align*}
& (e^{\wor}(H;P))^2 \\
& = \int_{[0,1]^{2s}}K(\bsx,\bsy)\, \rd \bsx\, \rd \bsy-\frac{2}{|P|}\sum_{\bsx\in P}\int_{[0,1]^s}K(\bsx,\bsy)\, \rd \bsy+\frac{1}{|P|^2}\sum_{\bsx,\bsy\in P}K(\bsx,\bsy),
\end{align*}
see for instance \cite{SW98}. Additionally if $K$ satisfies $\sum_{\bsk,\bsl\in \NN_0^s}|\hat{K}(\bsk,\bsl)|< \infty$, where $\hat{K}(\bsk,\bsl)$ denotes the $(\bsk,\bsl)$-th Walsh coefficient of $K$, i.e.,
\begin{align*}
\hat{K}(\bsk,\bsl) := \int_{[0,1]^{2s}}K(\bsx,\bsy)\overline{\wal_{\bsk}(\bsx)}\wal_{\bsl}(\bsy)\, \rd \bsx\, \rd \bsy,
\end{align*}
and if $\Pcal$ is a digital net in $G^s$, it holds from \cite[Proposition~19]{GSYxx} that
\begin{align*}
(e^{\wor}(H;\pi(\Pcal)))^2 = \sum_{\bsk,\bsl\in \Pcal^{\perp}\setminus \{\bszero\}}\hat{K}(\bsk,\bsl).
\end{align*}
Combining the above result with Lemma~\ref{lem:bani-dual-net}, we have the following.

\begin{theorem}\label{thm:QMC_error_RKHS}
Let $\Pcal$ be a digital net over $\ZZ_b$ in $G^s$ and $\Pcal^{\perp}$ its dual net. Let $H$ be a reproducing kernel Hilbert space with a continuous reproducing kernel $K:[0,1]^s\times [0,1]^s\to \RR$ which satisfies $\int_{[0,1]^s}\sqrt{K(\bsx,\bsx)}\, \rd \bsx< \infty$ and $\sum_{\bsk,\bsl\in \NN_0^s}|\hat{K}(\bsk,\bsl)|<\infty$. Then we have
\begin{align*}
\left( e^{\wor}(H;\pi(\Pcal_{\bant}))\right)^2 = \sum_{\substack{ \bsk,\bsl\in \Pcal^{\perp}\setminus \{\bszero\}\\ \delta(\bsk)= \delta(\bsl)= 0}}\hat{K}(\bsk,\bsl) .
\end{align*}
\end{theorem}

\begin{remark}\label{rem:QMC_error_RKHS}
Again, in general, we cannot expect a cancellation of $\hat{K}(\bsk,\bsl)$. Due to the triangle inequality, we have the following worst-case error bound
\begin{align*}
\left( e^{\wor}(H;\pi(\Pcal_{\bant}))\right)^2 \le \sum_{\substack{ \bsk,\bsl\in \Pcal^{\perp}\setminus \{\bszero\}\\ \delta(\bsk)= \delta(\bsl)= 0}}\lvert \hat{K}(\bsk,\bsl)\rvert.
\end{align*}
The right-hand side above is always less than or equal to $\sum_{\bsk,\bsl\in \Pcal^{\perp}\setminus \{\bszero\}}\lvert \hat{K}(\bsk,\bsl)\rvert$, which is a bound on $\left( e^{\wor}(H;\pi(\Pcal))\right)^2$.
\end{remark}

It can be seen from Theorems~\ref{thm:QMC_error_f} and \ref{thm:QMC_error_RKHS} that analyzing the Walsh coefficients play a central role in evaluating the integration error. We refer to \cite{BD09,Dick09,SYxx,Yoshikixx} and the references therein for recent results on the Walsh coefficients of smooth functions, some of which shall be used in the next section.
%%%%%%%%%%%%%%%%%%%%%%%%%%%%%%%%%%%%%%%%%%%%%%%%%%%%%%%%%%%
%%%%%%%%%%%%%%%%%%%%%%%%%%%%%%%%%%%%%%%%%%%%%%%%%%%%%%%%%%%
%%%%%%%%%%%%%%%%%%%%%%%%%%%%%%%%%%%%%%%%%%%%%%%%%%%%%%%%%%%
\section{Existence of good higher order polynomial lattices with antithetics}\label{sec:exist}
In this section, by using the result of Section~\ref{sec:banti}, we prove the existence of higher order polynomial lattice point sets with $b$-adic antithetics which achieve almost the optimal rate of convergence for smooth functions in weighted Sobolev spaces. For this purpose we first introduce weighted Sobolev spaces and higher order polynomial lattice point sets in Subsections~\ref{subsec:Sobolev} and \ref{subsec:HOPL}, respectively.
%%%%%%%%%%%%%%%%%%%%%%%%%%%%%%%%%%%%%%%%%%%%%%%%%%%%%%%%%%%
\subsection{Weighted Sobolev spaces}\label{subsec:Sobolev}
Here we introduce a weighted Sobolev space of smoothness $\alpha\in \NN$, $\alpha\ge 2$. Let us consider the one-dimensional unweighted case first. The Sobolev space which we consider is given by
\begin{align*}
 H_{\alpha} & := \Big\{f \colon [0,1]\to \RR \mid \\
 & \qquad f^{(r)} \colon \text{absolutely continuous for $r=0,\ldots,\alpha-1$}, f^{(\alpha)}\in L^2[0,1]\Big\},
\end{align*}
where $f^{(r)}$ denotes the $r$-th derivative of $f$. This space is indeed a reproducing kernel Hilbert space with an inner product $\langle \cdot, \cdot \rangle_{H_{\alpha}}$ and a reproducing kernel $K_{\alpha}\colon [0,1]\times [0,1]\to \RR$ given by 
\begin{align*}
 \langle f, g \rangle_{H_{\alpha}} = \sum_{r=0}^{\alpha-1}\int_{0}^{1}f^{(r)}(x)\, \rd x \int_{0}^{1}g^{(r)}(x)\, \rd x + \int_{0}^{1}f^{(\alpha)}(x)g^{(\alpha)}(x)\, \rd x,
\end{align*}
for $f,g\in H_{\alpha}$ and
\begin{align*}
 K_{\alpha}(x,y) = \sum_{r=0}^{\alpha}\frac{\Bcal_r(x)\Bcal_r(y)}{(r!)^2}+(-1)^{\alpha+1}\frac{\Bcal_{2\alpha}(|x-y|)}{(2\alpha)!} ,
\end{align*}
for $x,y\in [0,1]$, where $\Bcal_r$ denotes the Bernoulli polynomial of degree $r$.

Let us move on to the $s$-dimensional weighted case. In the following we write $\{1:n\}:=\{1,\ldots,n\}$ for $n\in \NN$. Let $\bsgamma=(\gamma_u)_{u\subseteq \{1:s\}}$ be a set of non-negative real numbers which are called weights. Note that the weights play a role in moderating the importance of different variables or groups of variables in function spaces \cite{SW98}. Now the weighted Sobolev space $H_{\alpha,\bsgamma}$ which we consider is a reproducing kernel Hilbert space whose inner product $\langle \cdot, \cdot \rangle_{H_{\alpha,\bsgamma}}$ and reproducing kernel $K_{\alpha,\bsgamma}\colon [0,1]^s\times [0,1]^s\to \RR$ are given as follows: 
\begin{align*}
 \langle f, g \rangle_{H_{\alpha,\bsgamma}} & = \sum_{u\subseteq \{1:s\}}\gamma_u^{-1}\sum_{v\subseteq u}\sum_{\bsr_{u\setminus v}\in \{1:\alpha-1\}^{|u\setminus v|}} \\
 & \qquad \times \int_{[0,1]^{|v|}}\left(\int_{[0,1]^{s-|v|}}f^{(\bsr_{u\setminus v},\bsalpha_v,\bszero)}(\bsx)\, \rd \bsx_{\{1:s\}\setminus v}\right) \\
 & \qquad \qquad \times \left(\int_{[0,1]^{s-|v|}} g^{(\bsr_{u\setminus v},\bsalpha_v,\bszero)}(\bsx) \, \rd \bsx_{\{1:s\}\setminus v}\right) \, \rd \bsx_v ,
\end{align*}
for $f,g\in H_{\alpha,\bsgamma}$ and
\begin{align*}
 K_{\alpha,\bsgamma}(\bsx,\bsy) = \sum_{u\subseteq \{1:s\}}\gamma_u \prod_{j\in u}\left\{\sum_{r=1}^{\alpha} \frac{\Bcal_r(x_j)\Bcal_r(y_j)}{(r!)^2}+(-1)^{\alpha+1}\frac{\Bcal_{2\alpha}(|x_j-y_j|)}{(2\alpha)!}\right\} ,
\end{align*}
for $\bsx=(x_1,\ldots,x_s),\bsy=(y_1,\ldots,y_s)\in [0,1]^s$. In the above, we use the following notation: For $v\subseteq \{1:s\}$ and $\bsx\in [0,1]^s$, let $\bsx_v=(x_j)_{j\in v}$. For $v\subseteq u\subseteq \{1:s\}$ and $\bsr_{u\setminus v}=(r_j)_{j\in u\setminus v}$, $(\bsr_{u\setminus v},\bsalpha_v,\bszero)$ denotes the $s$-dimensional vector whose $j$-th component is $r_j$ if $j\in u\setminus v$, $\alpha$ if $j\in v$, and $0$ otherwise. Note that the empty product always equals $1$ and we set $0/0=0$.

%%%%%%%%%%%%%%%%%%%%%%%%%%%%%%%%%%%%%%%%%%%%%%%%%%%%%%%%%%%
\subsection{Higher order polynomial lattice point sets}\label{subsec:HOPL}
We define higher order polynomial lattice point sets as digital nets in $G^s$. Note that they are originally defined as digital nets in $[0,1]^s$, whose construction is based on rational functions over finite fields \cite{DP07,Nied92}. In the remainder of this section, let $b$ be a prime.

We denote by $\ZZ_b[x]$ the set of all polynomials in $\ZZ_b$ and by $\ZZ_b((x^{-1}))$ the field of formal Laurent series in $\ZZ_b$. Every element of $\ZZ_b((x^{-1}))$ is given in the form $\sum_{l=w}^{\infty}t_lx^{-l}$ with some integer $w$ and $t_l\in \ZZ_b$. The definition of higher order polynomial lattice point sets is given as follows.
\begin{definition}
For $m,n\in \NN$ with $m\le n$, let $p\in \ZZ_b[x]$ with $\deg(p)=n$ and $\bsq=(q_1,\ldots,q_s)\in (\ZZ_b[x])^s$ with $\deg(q_j)<n$. For $1\le j\le s$, let us consider the expansion
\begin{align*}
\frac{q_j(x)}{p(x)} = \sum_{l=1}^{\infty}t_{l}^{(j)}x^{-l}\in \ZZ_b((x^{-1})).
\end{align*}
A higher order polynomial lattice point set in $G^s$ with modulus $p$ and generating vector $\bsq$, denoted by $\Pcal(p,\bsq)$, is a digital net over $\ZZ_b$ in $G^s$ with generating matrices $C_1,\ldots,C_s\in \ZZ_b^{\NN\times m}$, where each $C_j=(c_{l,r}^{(j)})$ is given by
\begin{align*}
c_{l,r}^{(j)} = \begin{cases}
t_{l+r-1}^{(j)} & \text{if $l\le n$,} \\
0 & \text{otherwise.}
\end{cases}
\end{align*}
\end{definition}

We shall often identify an non-negative integer $k=\kappa_0+\kappa_1 b+\cdots$ with a polynomial $k(x)=\kappa_0+\kappa_1 x+\cdots$. Moreover, for $n\in \NN$, the truncated polynomial $\tr_n(k)$ of $k$ is defined by 
\begin{align*}
\tr_n(k)(x) := \kappa_0+\kappa_1 x+\cdots+\kappa_{n-1}b^{n-1}.
\end{align*}
The following lemma gives another form of the dual net of $\Pcal(p,\bsq)$, see \cite[Lemma~15.25 \& Definition~15.26]{DPbook} for the proof.

\begin{lemma}\label{lem:dual_net_HOPL}
For $m,n\in \NN$ with $m\le n$, let $p\in \ZZ_b[x]$ with $\deg(p)=n$ and $\bsq=(q_1,\ldots,q_s)\in (\ZZ_b[x])^s$ with $\deg(q_j)<n$. The dual net of the higher order polynomial lattice point set $\Pcal(p,\bsq)$ is given by
\begin{align*}
\Pcal^{\perp}(p,\bsq) & = \big\{ \bsk=(k_1,\ldots,k_s)\in \NN_0^s \colon \\
                      & \qquad \tr_n(k_1)q_1+\cdots+\tr_n(k_s)q_s\equiv a \pmod p \; \text{with}\; \deg(a)<n-m \big\} .
\end{align*}
\end{lemma}

%%%%%%%%%%%%%%%%%%%%%%%%%%%%%%%%%%%%%%%%%%%%%%%%%%%%%%%%%%%
\subsection{Existence result}\label{subsec:exist_result}
We now prove the existence of good higher order polynomial lattice point sets with $b$-adic antithetics for QMC integration in $H_{\alpha,\bsgamma}$. More precisely, we prove the following theorem.
\begin{theorem}\label{thm:existence}
For an integer $\alpha \ge 2$ and a set of weights $\bsgamma$, let $H_{\alpha,\bsgamma}$ be the weighted Sobolev space. For $m,n\in \NN$ with $n\ge m$, let $p\in \ZZ_b[x]$ be irreducible with $\deg(p)=n$. Then there exists a generating vector $\bsq^*=(q^*_1,\ldots,q^*_s)\in (\ZZ_b[x])^s$ with $\deg(q^*_j)<n$ which satisfies
\begin{align*}
e^{\wor}(H_{\alpha,\bsgamma};\pi(\Pcal_{\ant}(p,\bsq^{*}))) \le \frac{1}{b^{\min(m/\lambda,2n)}}\left[\sum_{\emptyset \ne u\subseteq \{1,\ldots,s\}}\gamma_u^{\lambda/2}C_{\alpha,\lambda}^{|u|}\right]^{1/\lambda},
\end{align*}
for any $1/\alpha < \lambda \le 1$, where $C_{\alpha,\lambda}$ is positive and depends only on $\alpha$ and $\lambda$.
\end{theorem}

\begin{remark}
Let $n\ge \alpha m/2$. Then we have $\min(m/\lambda,2n)=m/\lambda$ for any $1/\alpha < \lambda \le 1$. From the above theorem and the fact that the number of points is given by $N=|\Pcal_{\ant}(p,\bsq^{*})|=b^{m+1}$, we have
\begin{align*}
e^{\wor}(H_{\alpha,\bsgamma};\pi(\Pcal_{\ant}(p,\bsq^{*}))) & \le \frac{1}{b^{m/\lambda}}\left[\sum_{\emptyset \ne u\subseteq \{1,\ldots,s\}}\gamma_u^{\lambda/2}C_{\alpha,\lambda}^{|u|}\right]^{1/\lambda} \\
& = \frac{1}{N^{1/\lambda}}\left[\sum_{\emptyset \ne u\subseteq \{1,\ldots,s\}}\gamma_u^{\lambda/2}C_{\alpha,\lambda}^{|u|}b\right]^{1/\lambda} ,
\end{align*}
for any $1/\alpha < \lambda \le 1$. Since we cannot achieve the convergence rate of the worst-case error of order $N^{-\alpha}$ in $H_{\alpha,\bsgamma}$ \cite{GSYxx2}, this result is almost optimal. Without $b$-adic antithetics, we need $n\ge \alpha (m+1)$ to achieve almost the optimal convergence rate of the worst-case error when the number of points is $b^{m+1}$ \cite{DP07}. This implies that we can find good point sets among a smaller number of candidates by the use of $b$-adic antithetics.

The case where $\alpha=2$ and $n=m$ seems particularly interesting. In this case we just have a classical polynomial lattice point set as introduced in \cite{Nied92}.
With the help of $b$-adic antithetics, it can achieve the convergence rate of order $N^{-2+\varepsilon}$ with arbitrarily small $\varepsilon>0$.
\end{remark}

In order to prove Theorem~\ref{thm:existence}, we need to introduce some more notation and some lemmas.

For $k\in \NN$, we denote its $b$-adic expansion by $k=\kappa_1b^{a_1-1}+\kappa_2b^{a_2-1}+\cdots +\kappa_vb^{a_v-1}$ with $a_1>a_2>\cdots >a_v>0$ and $\kappa_1,\ldots,\kappa_v\in \{1,\ldots,b-1\}$. Then we recall the definition of the function $\mu_{\alpha}:\NN_0\to \RR$ given by
\begin{align*}
\mu_{\alpha}(k) := \sum_{i=1}^{\min(v,\alpha)}a_i,
\end{align*}
and $\mu_{\alpha}(0)=0$, see \cite{Dick09}. For $\bsk=(k_1,\ldots,k_s)\in \NN_0^s$, we define
\begin{align*}
\mu_{\alpha}(\bsk) := \sum_{j=1}^{s}\mu_{\alpha}(k_j) .
\end{align*}
Regarding this function, we have the following result.
\begin{lemma}\label{lem:dick_weight}
Let $\alpha \ge 2$ be an integer. For $1/\alpha<\lambda\le 1$, let $A_{\alpha,\lambda}$ be given by
\begin{align*}
A_{\alpha,\lambda} = \sum_{v=1}^{\alpha-1}\prod_{i=1}^{v}\left( \frac{b-1}{b^{\lambda i}-1}\right)+\frac{b^{\lambda \alpha}-1}{b^{\lambda \alpha}-b}\prod_{i=1}^{\alpha}\left( \frac{b-1}{b^{\lambda i}-1}\right).
\end{align*}
The following holds true.
\begin{enumerate}
\item For any $1/\alpha<\lambda\le 1$, we have
  \begin{align*}
  \sum_{k=1}^{\infty}b^{-\lambda \mu_{\alpha}(k)}=A_{\alpha,\lambda} .
  \end{align*}
\item For any $1/\alpha<\lambda\le 1$ and $n\in \NN$, we have
  \begin{align*}
  \sum_{\substack{k=1\\ b^n\mid k}}^{\infty}b^{-\lambda \mu_{\alpha}(k)} \le \frac{A_{\alpha,\lambda}}{b^{\lambda n}} \quad \text{and}\quad \sum_{\substack{k=1\\ \delta(k)=0\\ b^n\mid k}}^{\infty}b^{-\lambda \mu_{\alpha}(k)}\le \frac{A_{\alpha,\lambda}}{b^{2\lambda n}}.
  \end{align*}
\end{enumerate}
\end{lemma}

\begin{proof}
Let us first consider Item~1 of the lemma. For $k\in \NN$, we denote its $b$-adic expansion by $k=\kappa_1b^{a_1-1}+\kappa_2b^{a_2-1}+\cdots +\kappa_vb^{a_v-1}$ with $a_1>a_2>\cdots >a_v>0$ and $\kappa_1,\ldots,\kappa_v\in \{1,\ldots,b-1\}$. We note that the value $\mu_{\alpha}(k)$ does not depend on $\kappa_1,\ldots,\kappa_v\in \{1,\ldots,b-1\}$. By arranging every element of $\NN$ according to the value of $v$ in its expansion, we have
  \begin{align}
  \sum_{k=1}^{\infty}b^{-\lambda \mu_{\alpha}(k)} & = \sum_{v=1}^{\infty}\sum_{a_1>\cdots > a_v>0}\sum_{\kappa_1,\ldots,\kappa_v\in \{1,\ldots,b-1\}}b^{-\lambda \mu_{\alpha}(\kappa_1b^{a_1-1}+\cdots +\kappa_vb^{a_v-1})} \nonumber \\
  & = \sum_{v=1}^{\infty}(b-1)^{v}\sum_{a_1>\cdots > a_v>0}b^{-\lambda \mu_{\alpha}(b^{a_1-1}+\cdots +b^{a_v-1})} \nonumber \\
  & = \sum_{v=1}^{\alpha-1}(b-1)^{v}\sum_{a_1>\cdots > a_v>0}b^{-\lambda (a_1+\cdots + a_v)} \label{eq:sum_weight1} \\
  & \qquad + \sum_{v=\alpha}^{\infty}(b-1)^{v}\sum_{a_1>\cdots > a_v>0}b^{-\lambda (a_1+\cdots + a_{\alpha})} . \label{eq:sum_weight2}
  \end{align}
As in the proof of \cite[Lemma~25]{GSY15} in which $2\lambda$ should be replaced by $\lambda$ here, for the inner sum of (\ref{eq:sum_weight1}) we have
  \begin{align*}
  \sum_{a_1>\cdots > a_v>0}b^{-\lambda (a_1+\cdots + a_v)} = \prod_{i=1}^{v}\left( \frac{1}{b^{\lambda i}-1}\right),
  \end{align*}
for any $0<\lambda\le 1$. Similarly for the double sum of (\ref{eq:sum_weight2}) we have
  \begin{align*}
  \sum_{v=\alpha}^{\infty}(b-1)^{v}\sum_{a_1>\cdots > a_v>0}b^{-\lambda (a_1+\cdots + a_{\alpha})} = \frac{b^{\lambda \alpha}-1}{b^{\lambda \alpha}-b}\prod_{i=1}^{\alpha}\left( \frac{b-1}{b^{\lambda i}-1}\right),
  \end{align*}
for any $1/\alpha<\lambda\le 1$. Here we note that the condition $\lambda > 1/\alpha$ is required for this double sum to be finite. Thus the result for Item~1 follows.

Let us move on to the first part of Item~2 of the lemma. If $b^n\mid k$ holds, $k$ is given in the form $lb^n$ for $l\in \NN$. Following an argument similar to the proof of Item~1, for any $1/\alpha<\lambda\le 1$ we have
  \begin{align*}
  \sum_{\substack{k=1\\ b^n\mid k}}^{\infty}b^{-\lambda \mu_{\alpha}(k)} & = \sum_{l=1}^{\infty}b^{-\lambda \mu_{\alpha}(lb^n)} \\
  & = \sum_{v=1}^{\infty}\sum_{a_1>\cdots > a_v>0}\sum_{\kappa_1,\ldots,\kappa_v\in \{1,\ldots,b-1\}}b^{-\lambda \mu_{\alpha}(\kappa_1b^{a_1+n-1}+\cdots +\kappa_vb^{a_v+n-1})} \nonumber \\
  & = \sum_{v=1}^{\infty}(b-1)^{v}\sum_{a_1>\cdots > a_v>0}b^{-\lambda \mu_{\alpha}(b^{a_1+n-1}+\cdots +b^{a_v+n-1})} \nonumber \\
  & = \sum_{v=1}^{\alpha-1}(b-1)^{v}b^{-\lambda vn}\sum_{a_1>\cdots > a_v>0}b^{-\lambda (a_1+\cdots + a_v)} \\
  & \qquad + b^{-\lambda \alpha n}\sum_{v=\alpha}^{\infty}(b-1)^{v}\sum_{a_1>\cdots > a_v>0}b^{-\lambda (a_1+\cdots + a_{\alpha})} \\
  & = \sum_{v=1}^{\alpha-1}\frac{1}{b^{\lambda vn}}\prod_{i=1}^{v}\left( \frac{b-1}{b^{\lambda i}-1}\right)+ \frac{1}{b^{\lambda \alpha n}}\frac{b^{\lambda \alpha}-1}{b^{\lambda \alpha}-b}\prod_{i=1}^{\alpha}\left( \frac{b-1}{b^{\lambda i}-1}\right) \\
  & \le \frac{1}{b^{\lambda n}}\left[ \sum_{v=1}^{\alpha-1}\prod_{i=1}^{v}\left( \frac{b-1}{b^{\lambda i}-1}\right)+ \frac{b^{\lambda \alpha}-1}{b^{\lambda \alpha}-b}\prod_{i=1}^{\alpha}\left( \frac{b-1}{b^{\lambda i}-1}\right) \right] = \frac{A_{\alpha,\lambda}}{b^{\lambda n}},
  \end{align*}
where the last inequality stems from the condition $\alpha \ge 2$. Thus the result for the first part of Item~2 follows.

Finally let us consider the second part of Item~2 of the lemma. Again if $b^n\mid k$ holds, $k$ is given in the form $lb^n$ for $l\in \NN$. Moreover, we have $\delta(lb^n)=\delta(l)$ for any $l\in \NN$, and if $\delta(l)=0$ holds, the $b$-adic expansion of $l$ has to contain at least two non-zero digits. Thus for any $1/\alpha<\lambda\le 1$ we have
  \begin{align*}
  \sum_{\substack{k=1\\ \delta(k)=0\\ b^n\mid k}}^{\infty}b^{-\lambda \mu_{\alpha}(k)} & = \sum_{\substack{l=1\\ \delta(l)=0}}^{\infty}b^{-\lambda \mu_{\alpha}(lb^n)} \\
  & \le \sum_{v=2}^{\infty}\sum_{a_1>\cdots > a_v>0}\sum_{\kappa_1,\ldots,\kappa_v\in \{1,\ldots,b-1\}}b^{-\lambda \mu_{\alpha}(\kappa_1b^{a_1+n-1}+\cdots +\kappa_vb^{a_v+n-1})} \nonumber \\
  & = \sum_{v=2}^{\alpha-1}\frac{1}{b^{\lambda vn}}\prod_{i=1}^{v}\left( \frac{b-1}{b^{\lambda i}-1}\right)+ \frac{1}{b^{\lambda \alpha n}}\frac{b^{\lambda \alpha}-1}{b^{\lambda \alpha}-b}\prod_{i=1}^{\alpha}\left( \frac{b-1}{b^{\lambda i}-1}\right) \\
  & \le \frac{1}{b^{2\lambda n}}\left[ \sum_{v=2}^{\alpha-1}\prod_{i=1}^{v}\left( \frac{b-1}{b^{\lambda i}-1}\right)+ \frac{b^{\lambda \alpha}-1}{b^{\lambda \alpha}-b}\prod_{i=1}^{\alpha}\left( \frac{b-1}{b^{\lambda i}-1}\right) \right] \le \frac{A_{\alpha,\lambda}}{b^{2\lambda n}}.
  \end{align*}
Thus the result for the second part of Item~2 follows.
\end{proof}

Since the reproducing kernel $K_{\alpha,\bsgamma}$ is continuous and satisfies the conditions $\int_{[0,1]^s}\sqrt{K_{\alpha,\bsgamma}(\bsx,\bsx)}\, \rd \bsx< \infty$ and $\sum_{\bsk,\bsl\in \NN_0^s}|\hat{K}_{\alpha,\bsgamma}(\bsk,\bsl)|<\infty$ as shown in \cite[Section~4.1]{GSYxx}, we can apply Theorem~\ref{thm:QMC_error_RKHS}. Using the bound on the Walsh coefficients $\hat{K}_{\alpha,\bsgamma}(\cdot,\cdot)$ shown by Baldeaux and Dick in \cite[Section~3.1]{BD09} together with the triangle inequality, we have the following. Since the proof is almost the same as that used in \cite[Theorem~23]{GSYxx}, we omit it.

\begin{lemma}\label{lem:bound_worst-case}
For an integer $\alpha \ge 2$ and a set of weights $\bsgamma$, let $H_{\alpha,\bsgamma}$ be the weighted Sobolev space. For $m,n\in \NN$ with $n\ge m$, let $p\in \ZZ_b[x]$ with $\deg(p)=n$ and $\bsq=(q_1,\ldots,q_s)\in (\ZZ_b[x])^s$ with $\deg(q_j)<n$. Then the worst-case error of $\pi(\Pcal_{\ant}(p,\bsq^{*}))$ in $H_{\alpha,\bsgamma}$ can be bounded by
\begin{align*}
e^{\wor}(H_{\alpha,\bsgamma};\pi(\Pcal_{\ant}(p,\bsq^{*}))) \le \sum_{\emptyset \ne u\subseteq \{1:s\}}\gamma_u^{1/2}D_{\alpha}^{|u|/2}\sum_{\substack{ \bsk_u\in \NN^{|u|}\\ (\bsk_u,\bszero)\in \Pcal^{\perp}(p,\bsq)\\ \delta(\bsk_u)=0}}b^{-\mu_{\alpha}(\bsk_u)},
\end{align*}
where $D_{\alpha}>0$ is given by
\begin{align*}
D_{\alpha} = \max_{1\le v\le \alpha}\left\{ \sum_{\tau=v}^{\alpha}\frac{C_{\tau}^2}{b^{2(\tau-v)}}+ \frac{2C_{2\alpha}}{b^{2(\alpha-v)}}\right\},
\end{align*}
with
\begin{align*}
C_{1} & = \frac{1}{2\sin(\pi/b))} \quad \text{and}\quad C_{\tau} = \frac{(1+1/b+1/(b(b+1)))^{\tau-2}}{(2\sin(\pi/b))^{\tau}} \quad \text{for $\tau\ge 2$}.
\end{align*}
\end{lemma}
\noindent
In the following, we simply write the bound on $e^{\wor}(H_{\alpha,\bsgamma};\pi(\Pcal_{\ant}(p,\bsq^{*})))$ given in the above lemma as
\begin{align*}
B_{\alpha,\bsgamma}(p,\bsq) = \sum_{\emptyset \ne u\subseteq \{1:s\}}\gamma_u^{1/2}D_{\alpha}^{|u|/2}\sum_{\substack{ \bsk_u\in \NN^{|u|}\\ (\bsk_u,\bszero)\in \Pcal^{\perp}(p,\bsq)\\ \delta(\bsk_u)=0}}b^{-\mu_{\alpha}(\bsk_u)}.
\end{align*}

Now we are ready to prove Theorem~\ref{thm:existence}. 

\begin{proof}[Proof of Theorem~\ref{thm:existence}]
Let us define
\begin{align*}
R_{b,n} := \{q\in \ZZ_b[x]\colon \deg(q)<n \} ,
\end{align*}
and let $\bsq^{*}$ in Theorem~\ref{thm:existence} be given by
\begin{align*}
\bsq^{*} = \argmin_{\bsq\in R_{b,n}^s}B_{\alpha,\bsgamma}(p,\bsq).
\end{align*}
Due to the averaging argument and Jensen's inequality $(\sum_{k}a_k)^{\lambda}\le \sum_{k}a_k^{\lambda}$ for $0<\lambda\le 1$ with $a_k\ge 0$, we have
\begin{align*}
\left(B_{\alpha,\bsgamma}(p,\bsq^{*})\right)^{\lambda} & \le \frac{1}{b^{ns}}\sum_{\bsq\in R_{b,n}^s}\left(B_{\alpha,\bsgamma}(p,\bsq)\right)^{\lambda} \\
& \le \frac{1}{b^{ns}}\sum_{\bsq\in R_{b,n}^s}\sum_{\emptyset \ne u\subseteq \{1,\ldots,s\}}\gamma_u^{\lambda/2}D_{\alpha}^{\lambda |u|/2}\sum_{\substack{ \bsk_u\in \NN^{|u|}\\ (\bsk_u,\bszero)\in P^{\perp}(p,\bsq)\\ \delta(\bsk_u)=0}}b^{-\lambda\mu_{\alpha}(\bsk_u)} \\
& = \sum_{\emptyset \ne u\subseteq \{1,\ldots,s\}}\gamma_u^{\lambda/2}D_{\alpha}^{\lambda |u|/2}\sum_{\substack{ \bsk_u\in \NN^{|u|}\\ \delta(\bsk_u)=0}}b^{-\lambda\mu_{\alpha}(\bsk_u)} \\
& \qquad \times \frac{1}{b^{n|u|}}\sum_{\substack{\bsq_u\in R_{b,n}^{|u|} \\ \tr_n(\bsk_u)\cdot\bsq_u \equiv a \pmod p\\ \deg(a)<n-m}}1 ,
\end{align*}
for any $0<\lambda\le 1$. From the result shown in \cite[Section~4]{DP07}, the innermost sum of the last expression is given by
\begin{align*}
\frac{1}{b^{n|u|}}\sum_{\substack{\bsq_u\in R_{b,n}^{|u|} \\ \tr_n(\bsk_u)\cdot\bsq_u \equiv a \pmod p\\ \deg(a)<n-m}}1 = \begin{cases}
1 & \text{if $b^n\mid k_j$ for all $j\in u$,}\\
\frac{1}{b^m} & \text{otherwise.} \\
\end{cases}
\end{align*}
Thus we have
\begin{align}
\left(B_{\alpha,\bsgamma}(p,\bsq^{*})\right)^{\lambda} & \le \sum_{\emptyset \ne u\subseteq \{1,\ldots,s\}}\gamma_u^{\lambda/2}D_{\alpha}^{\lambda |u|/2}\sum_{\substack{ \bsk_u\in \NN^{|u|}\\ \delta(\bsk_u)=0 \\ b^n \mid k_j, \forall j\in u}}b^{-\lambda\mu_{\alpha}(\bsk_u)} \nonumber \\
& \qquad + \frac{1}{b^m}\sum_{\emptyset \ne u\subseteq \{1,\ldots,s\}}\gamma_u^{\lambda/2}D_{\alpha}^{\lambda |u|/2}\sum_{\substack{ \bsk_u\in \NN^{|u|}\\ \delta(\bsk_u)=0 \\ \exists j\in u\colon b^n \nmid k_j}}b^{-\lambda\mu_{\alpha}(\bsk_u)}. \label{eq:main_proof1}
\end{align}
In the following, let $1/\alpha <\lambda\le 1$. The inner sum of the first term on the right-hand side of (\ref{eq:main_proof1}) is bounded above as follows: For $u\subseteq \{1,\ldots,s\}$ with $|u|\ge 2$, we have
\begin{align*}
\sum_{\substack{ \bsk_u\in \NN^{|u|}\\ \delta(\bsk_u)=0 \\ b^n \mid k_j, \forall j\in u}}b^{-\lambda\mu_{\alpha}(\bsk_u)} & \le \sum_{\substack{ \bsk_u\in \NN^{|u|}\\ b^n \mid k_j, \forall j\in u}}b^{-\lambda\mu_{\alpha}(\bsk_u)} \\
& = \left[ \sum_{\substack{k\in \NN\\ b^n \mid k}}b^{-\lambda\mu_{\alpha}(k)} \right]^{|u|} \le \frac{A_{\alpha,\lambda}^{|u|}}{b^{\lambda |u|n}}\le \frac{A_{\alpha,\lambda}^{|u|}}{b^{2\lambda n}} ,
\end{align*}
where we used the first part of Item~2 in Lemma~\ref{lem:dick_weight} in the second inequality. For $u\subseteq \{1,\ldots,s\}$ with $|u|=1$, by using the second part of Item~2 in Lemma~\ref{lem:dick_weight}, we have
\begin{align*}
\sum_{\substack{ k\in \NN \\ \delta(k)=0 \\ b^n \mid k}}b^{-\lambda\mu_{\alpha}(k)} \le \frac{A_{\alpha,\lambda}}{b^{2\lambda n}}.
\end{align*}
By using Item~1 in Lemma~\ref{lem:dick_weight}, the inner sum of the second term on the right-hand side of (\ref{eq:main_proof1}) can be bounded by
\begin{align*}
\sum_{\substack{ \bsk_u\in \NN^{|u|}\\ \delta(\bsk_u)=0 \\ \exists j\in u\colon b^n \nmid k_j}}b^{-\lambda\mu_{\alpha}(\bsk_u)} \le \sum_{\bsk_u\in \NN^{|u|}}b^{-\lambda\mu_{\alpha}(\bsk_u)} = \left[ \sum_{k\in \NN}b^{-\lambda\mu_{\alpha}(k)} \right]^{|u|} = A_{\alpha,\lambda}^{|u|} .
\end{align*}
Since we now have
\begin{align*}
\left(B_{\alpha,\bsgamma}(p,\bsq^{*})\right)^{\lambda} & \le \frac{1}{b^{2\lambda n}} \sum_{\emptyset \ne u\subseteq \{1,\ldots,s\}}\gamma_u^{\lambda/2}D_{\alpha}^{\lambda |u|/2}A_{\alpha,\lambda}^{|u|} \\
& \qquad + \frac{1}{b^m}\sum_{\emptyset \ne u\subseteq \{1,\ldots,s\}}\gamma_u^{\lambda/2}D_{\alpha}^{\lambda |u|/2}A_{\alpha,\lambda}^{|u|} \\
& \le \frac{1}{b^{\min(m,2\lambda n)}}\sum_{\emptyset \ne u\subseteq \{1,\ldots,s\}}2\gamma_u^{\lambda/2}D_{\alpha}^{\lambda |u|/2}A_{\alpha,\lambda}^{|u|},
\end{align*}
the worst-case error of QMC integration using $\pi(\Pcal_{\ant}(p,\bsq^{*}))$ can be bounded by
\begin{align*}
e^{\wor}(H_{\alpha,\bsgamma};\pi(\Pcal_{\ant}(p,\bsq^{*}))) & \le B_{\alpha,\bsgamma}(p,\bsq^{*}) \\
& \le \frac{1}{b^{\min(m/\lambda, 2n)}}\left[\sum_{\emptyset \ne u\subseteq \{1,\ldots,s\}}2\gamma_u^{\lambda/2}D_{\alpha}^{\lambda |u|/2}A_{\alpha,\lambda}^{|u|}\right]^{1/\lambda},
\end{align*}
which completes the proof by setting $C_{\alpha,\lambda}=2D_{\alpha}^{\lambda /2}A_{\alpha,\lambda}$.
\end{proof}
%%%%%%%%%%%%%%%%%%%%%%%%%%%%%%%%%%%%%%%%%%%%%%%%%%%%%%%%%%%
%%%%%%%%%%%%%%%%%%%%%%%%%%%%%%%%%%%%%%%%%%%%%%%%%%%%%%%%%%%
%%%%%%%%%%%%%%%%%%%%%%%%%%%%%%%%%%%%%%%%%%%%%%%%%%%%%%%%%%%
\section{Numerical experiments}\label{sec:numer}
Finally we conduct some numerical experiments up to $s=100$ based on Sobol' point sets, which are a special construction of digital nets over $\ZZ_2$.
Our purpose here is to compare the performances of Sobol' point sets with and without dyadic antithetics.
As can be expected from Remarks~\ref{rem:QMC_error_f} and \ref{rem:QMC_error_RKHS}, Sobol' point sets with antithetics may perform well at least as compared to those without antithetics.
We consider the following three test functions:
\begin{align*}
f_1(\bsx) & = \exp\left( \theta \sum_{j=1}^{s}\frac{x_j}{j^{\zeta}}\right) \; \text{for $\theta,\zeta>0$},\\
f_2(\bsx) & = \prod_{j=1}^{s}\left( 1+\frac{w^{j}}{21}\left( -10+42x_j^2-42x_j^5+21x_j^6\right)\right) \; \text{for $w>0$}, \\
f_3(\bsx) & = \prod_{j=1}^{s}\Big( 1+\frac{w^{j}}{8}( 31-84x_j^2+8x_j^3+70x_j^4-28x_j^6+8x_j^7 \\
          & \quad \qquad  -16\cos(1)-16\sin(x_j))\Big) \; \text{for $w>0$}.
\end{align*}
The first one was used in \cite{GSxx}, whereas the latter two were in \cite{DNP14}.
We choose these smooth functions so that the conditions (continuity and summability of the Walsh coefficients) on $f$ given in Theorem~\ref{thm:QMC_error_f} are satisfied.
The parameters $\xi$ (in $f_1$) and $w$ (in $f_2$ and $f_3$) play a role in moderating the importance of different variables or groups of variables. Since $I(f_i)$ is known exactly for all $i=1,2,3$, we consider the absolute integration error $|I(f_i;P)-I(f_i)|$.

Figure~\ref{fig:1} shows the absolute integration errors for $f_1$ as functions of the number of points with $\theta=0.1$ and $(s,\zeta)=(10,1),(10,2),(100,1),(100,2)$. In each graph, the line marked by triangle represents the integration error when using Sobol' point sets without antithetics. For all the cases, the error converges almost exactly with order $N^{-1}$. The line marked by circle represents the integration error when using Sobol' point sets with antithetics. For all the cases, the error converges with order around $N^{-1.35}$, which is faster than $N^{-1}$.

Figure~\ref{fig:2} shows the absolute integration errors for $f_2$ and $f_3$ as functions of the number of points with $s=100$ and $\omega=0.5,0.1$. Again for all the cases, the error when using Sobol' point sets without antithetics converges almost exactly with order $N^{-1}$. Regardless of $f_2$ or $f_3$, the error when using Sobol' point sets with antithetics converges with order around $N^{-1.35}$ for $\omega=0.5$, whereas it converges with order around $N^{-1.60}$ for $\omega=0.1$. For the case $\omega=0.1$, the erratic convergence behavior is observed for Sobol' point sets with antithetics, which is though elusive and requires further work to address.

%%%%%%%%%%%%%%%%%%%%%%%%%%%%%%%%%%%%%%%%%%%%%%%%%%%%%%%%%%%
%%%%%%%%%%%%%%%%%%%%%%%%%%%%%%%%%%%%%%%%%%%%%%%%%%%%%%%%%%%
%%%%%%%%%%%%%%%%%%%%%%%%%%%%%%%%%%%%%%%%%%%%%%%%%%%%%%%%%%%
\section*{Acknowledgments}
This work was supported by JSPS Grant-in-Aid for Young Scientists No.15K20964.

\begin{figure}
\begin{center}
\includegraphics[width=0.45\textwidth]{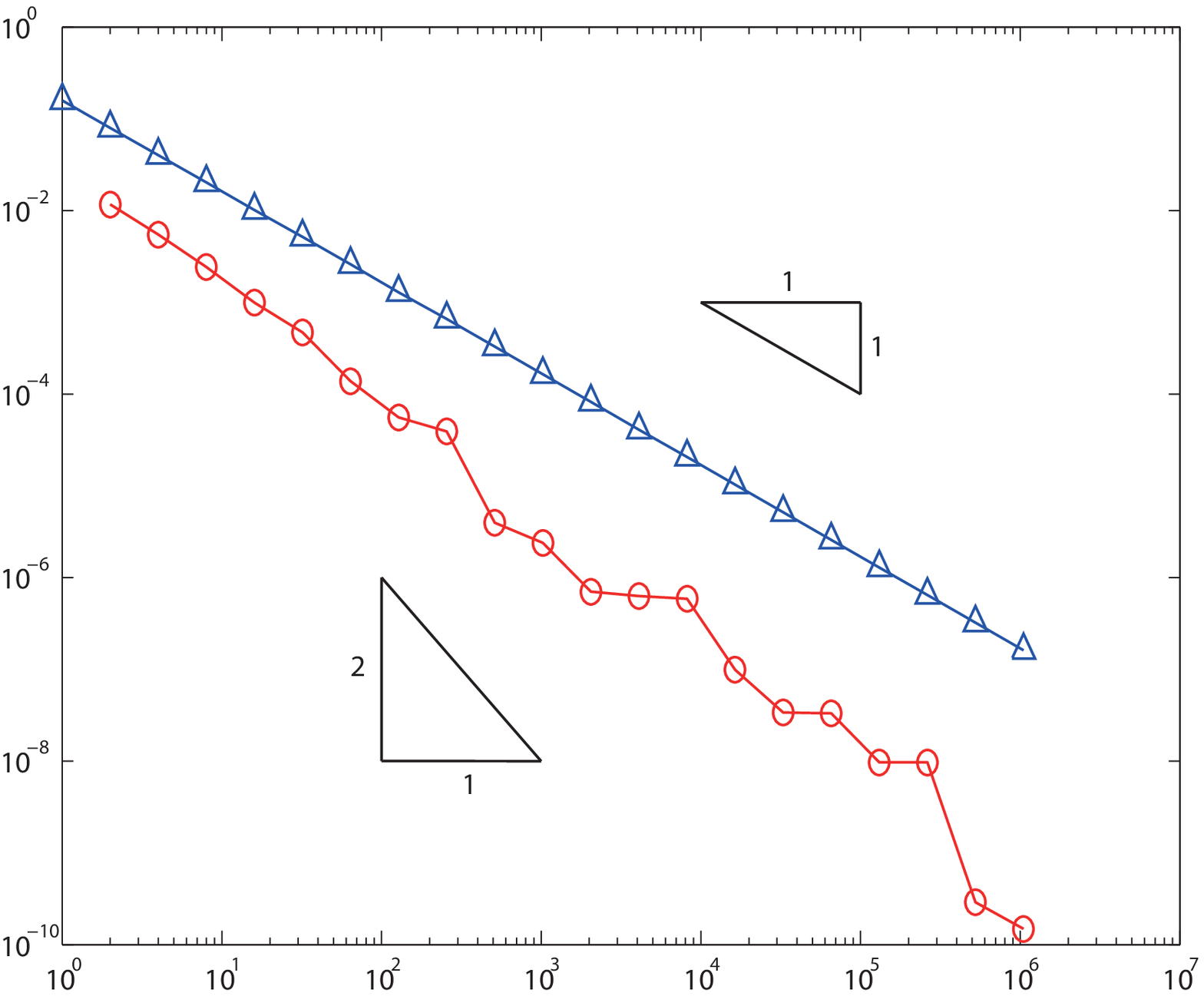}
\includegraphics[width=0.45\textwidth]{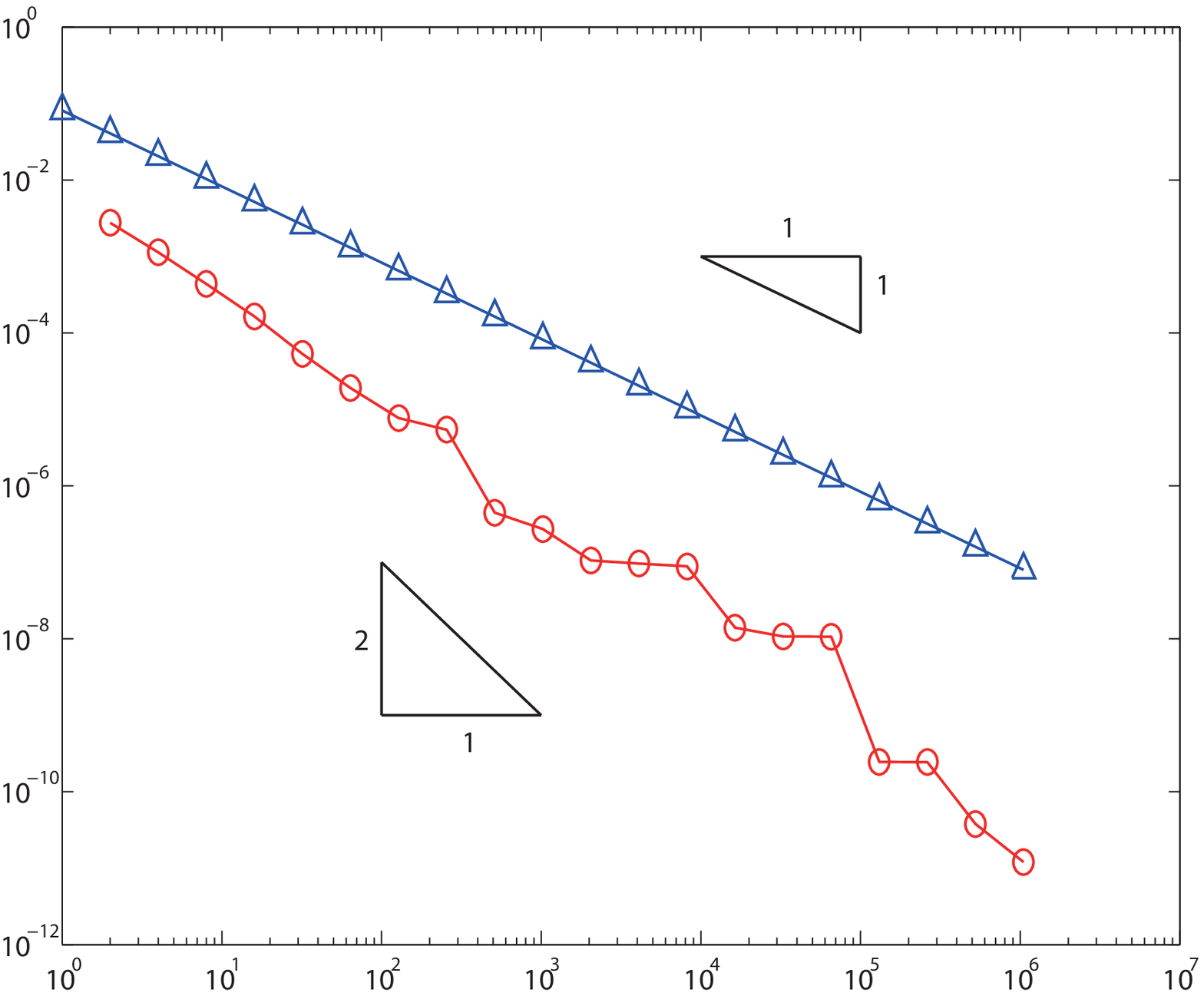}\\
\includegraphics[width=0.45\textwidth]{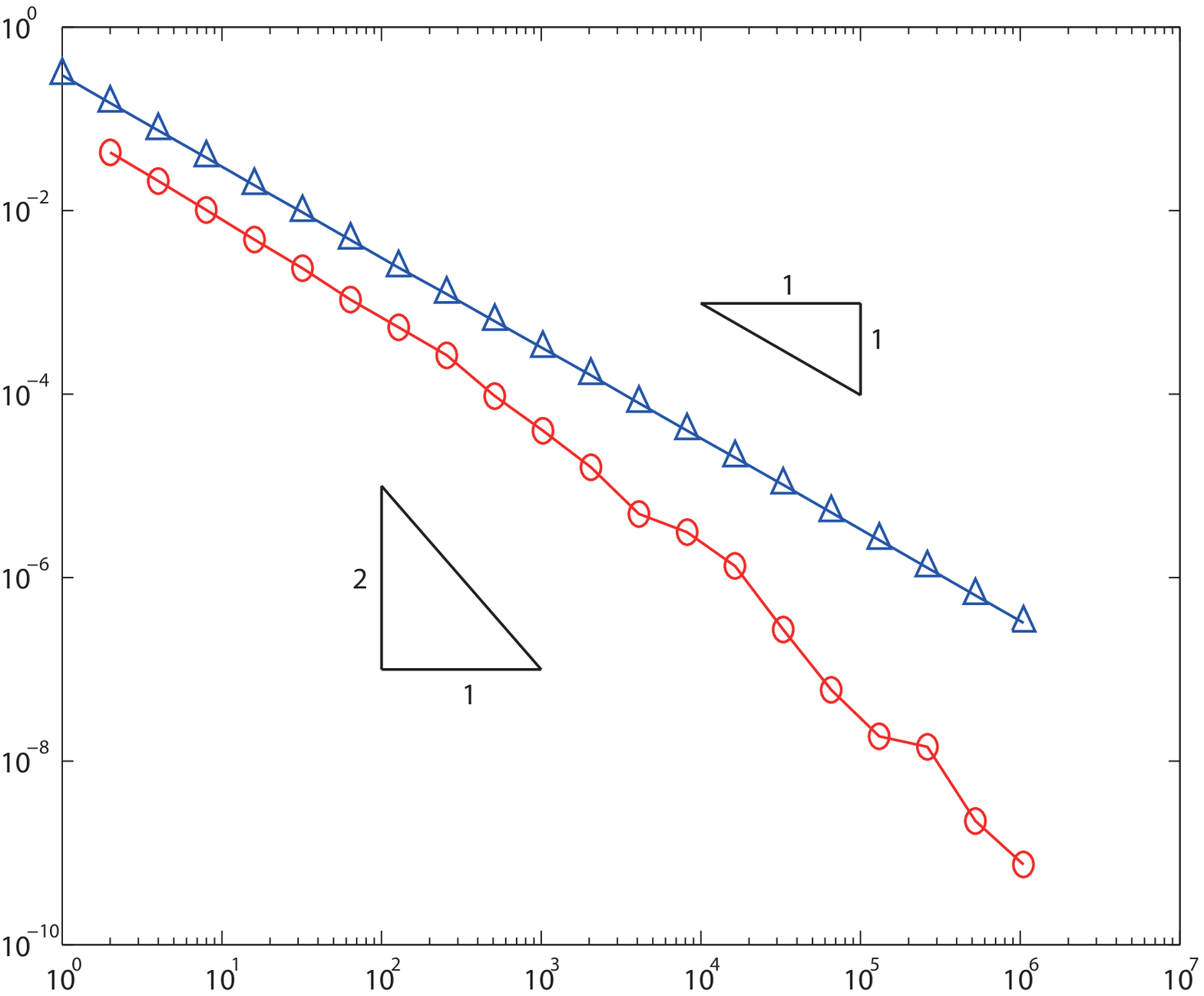}
\includegraphics[width=0.45\textwidth]{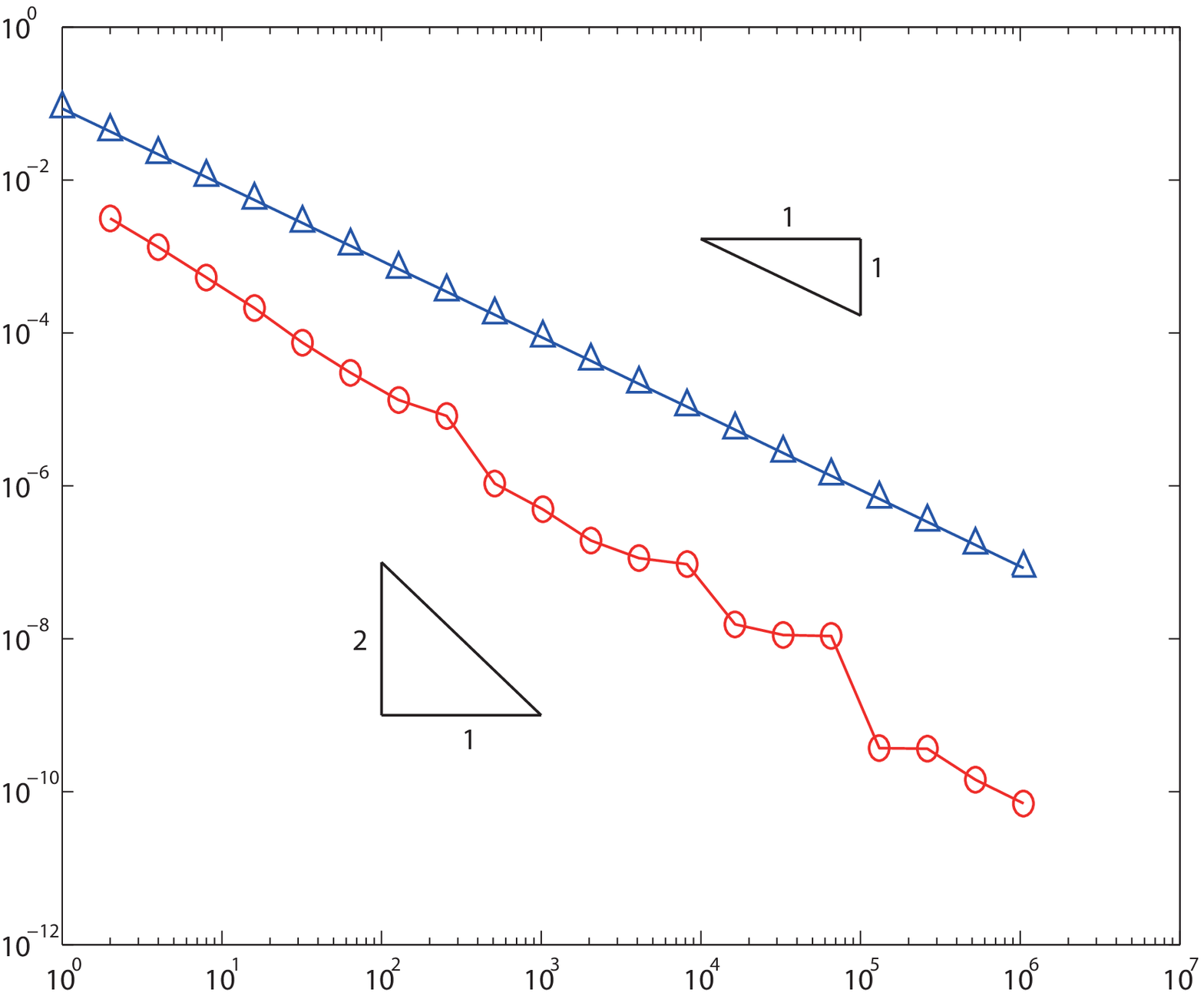}
\caption{Absolute integration error for $f_1$ vs number of points. The lines marked by circle and triangle represent the errors using QMC rules based on a Sobol' point set with and without dyadic antithetics, respectively. (Left top: $s=10, \zeta=1$, right top: $s=10, \zeta=2$, left bottom: $s=100, \zeta=1$, right bottom: $s=100, \zeta=2$)}
\label{fig:1}
\end{center}
\end{figure}

\begin{figure}
\begin{center}
\includegraphics[width=0.45\textwidth]{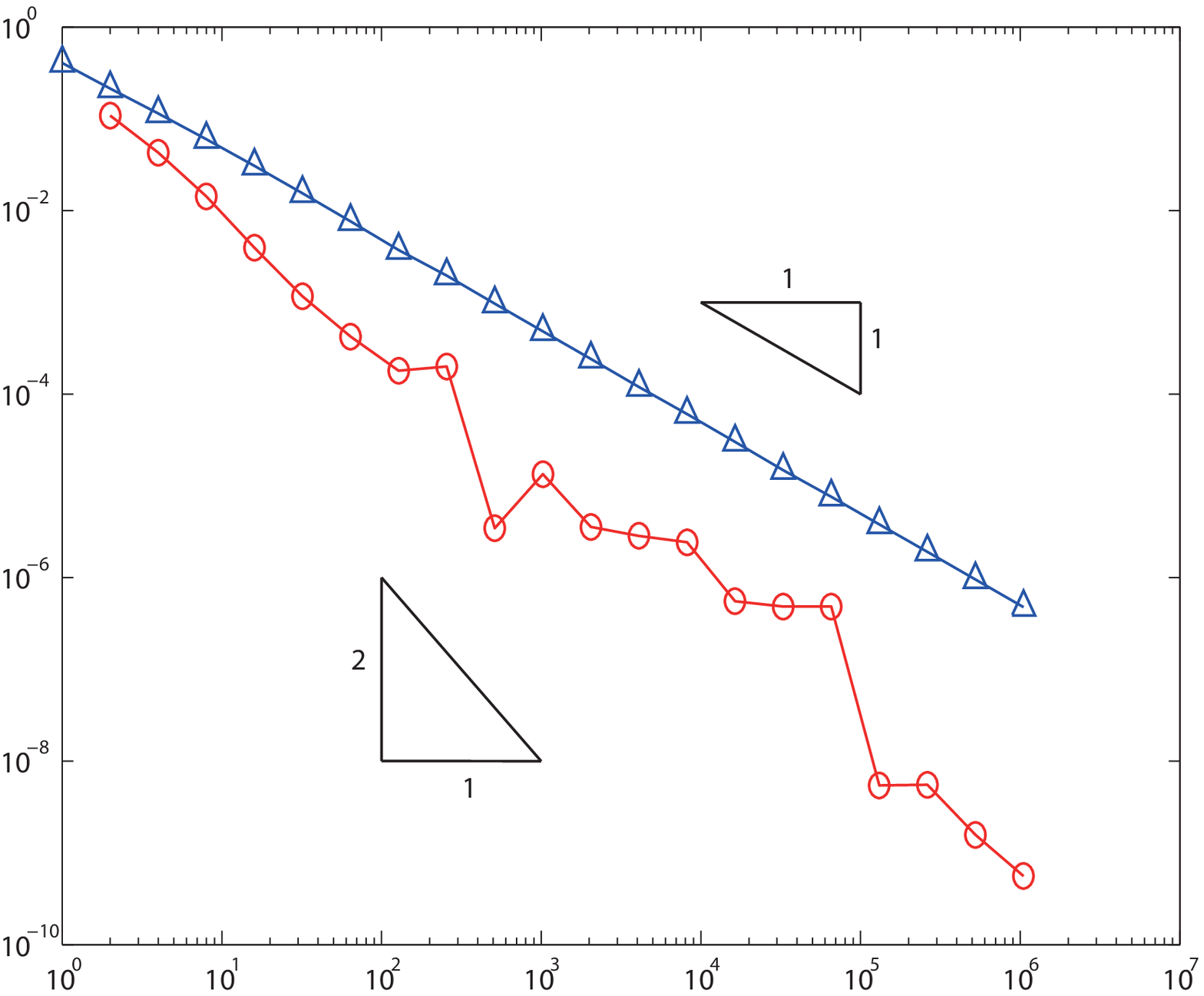}
\includegraphics[width=0.45\textwidth]{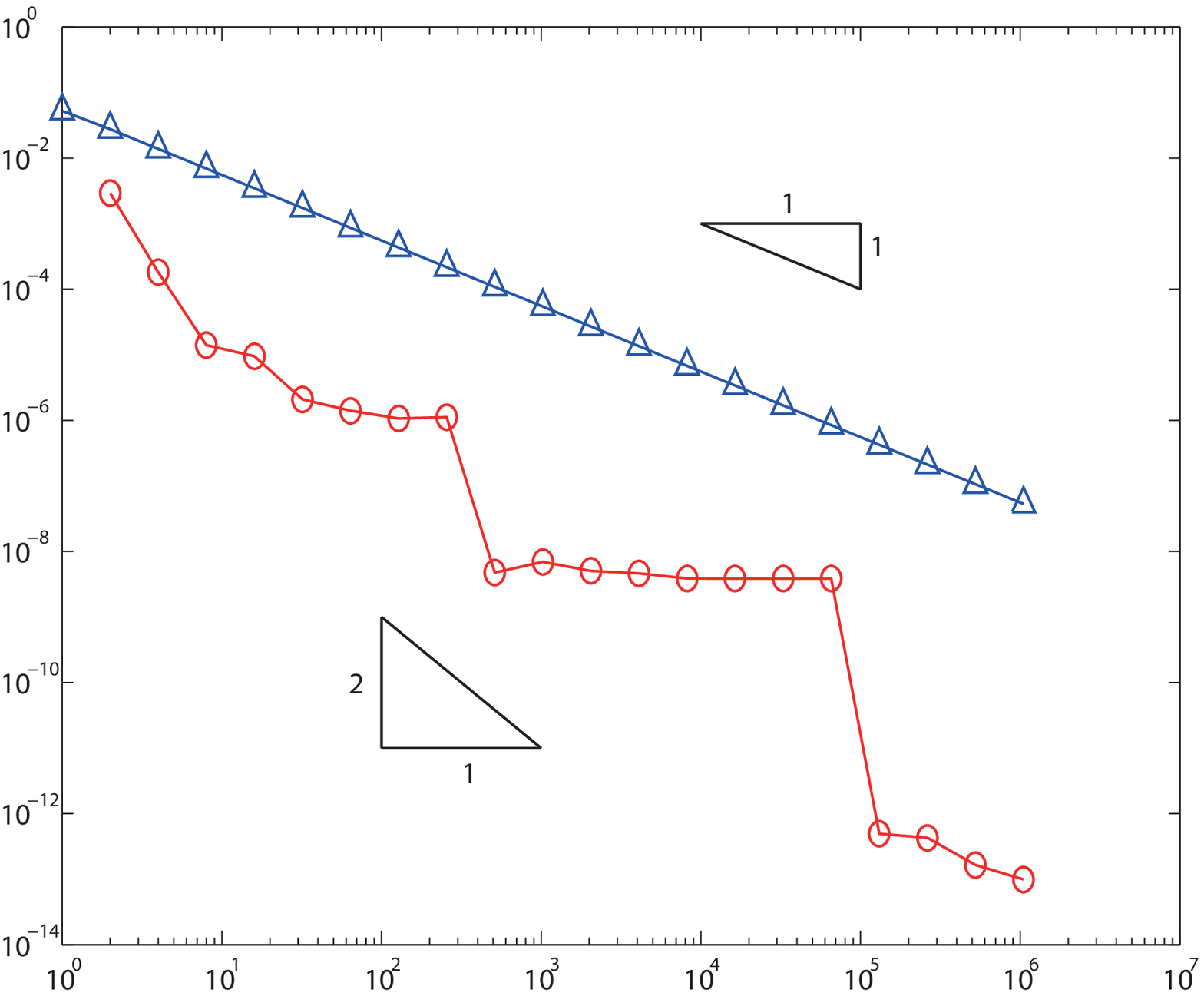}\\
\includegraphics[width=0.45\textwidth]{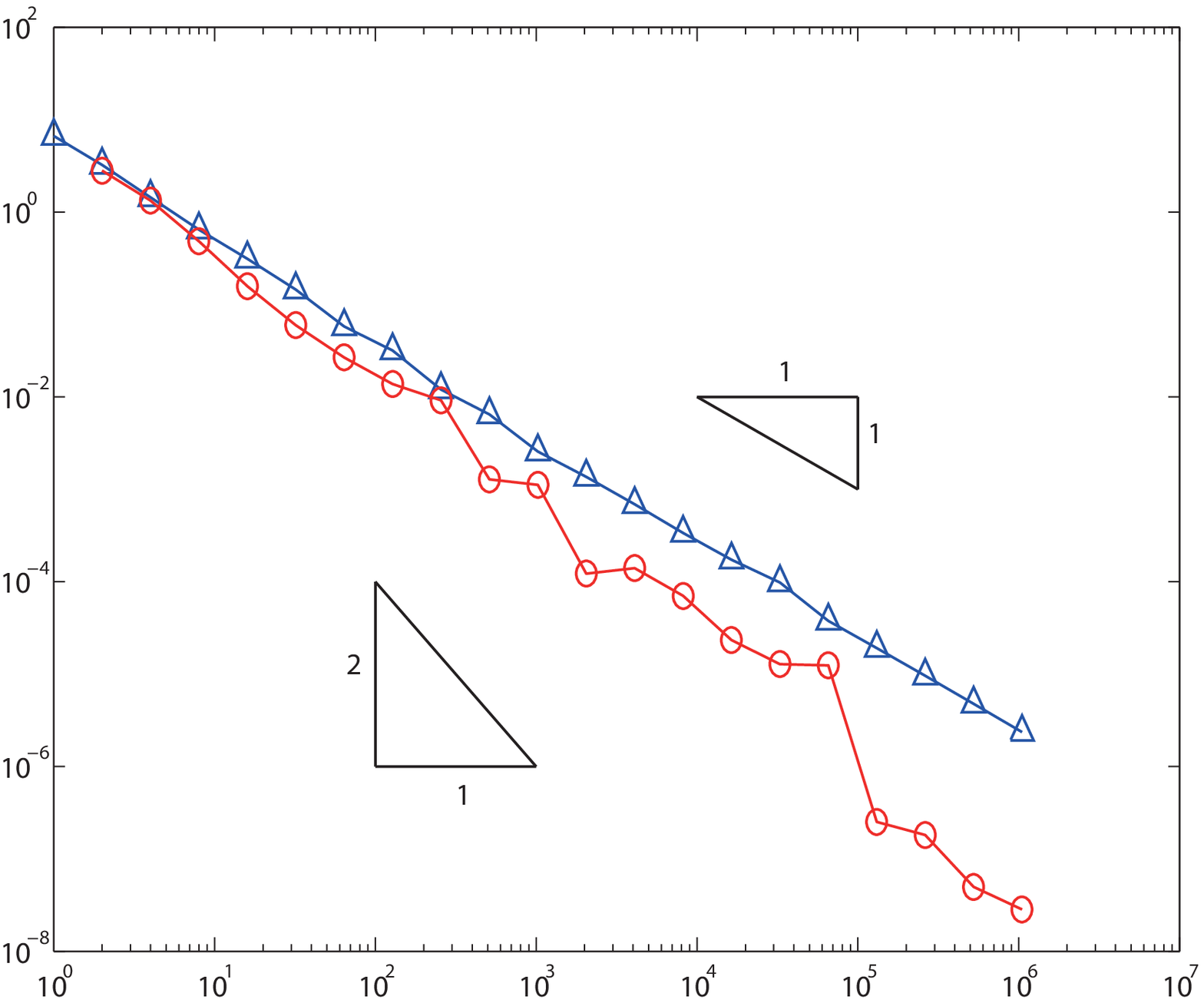}
\includegraphics[width=0.45\textwidth]{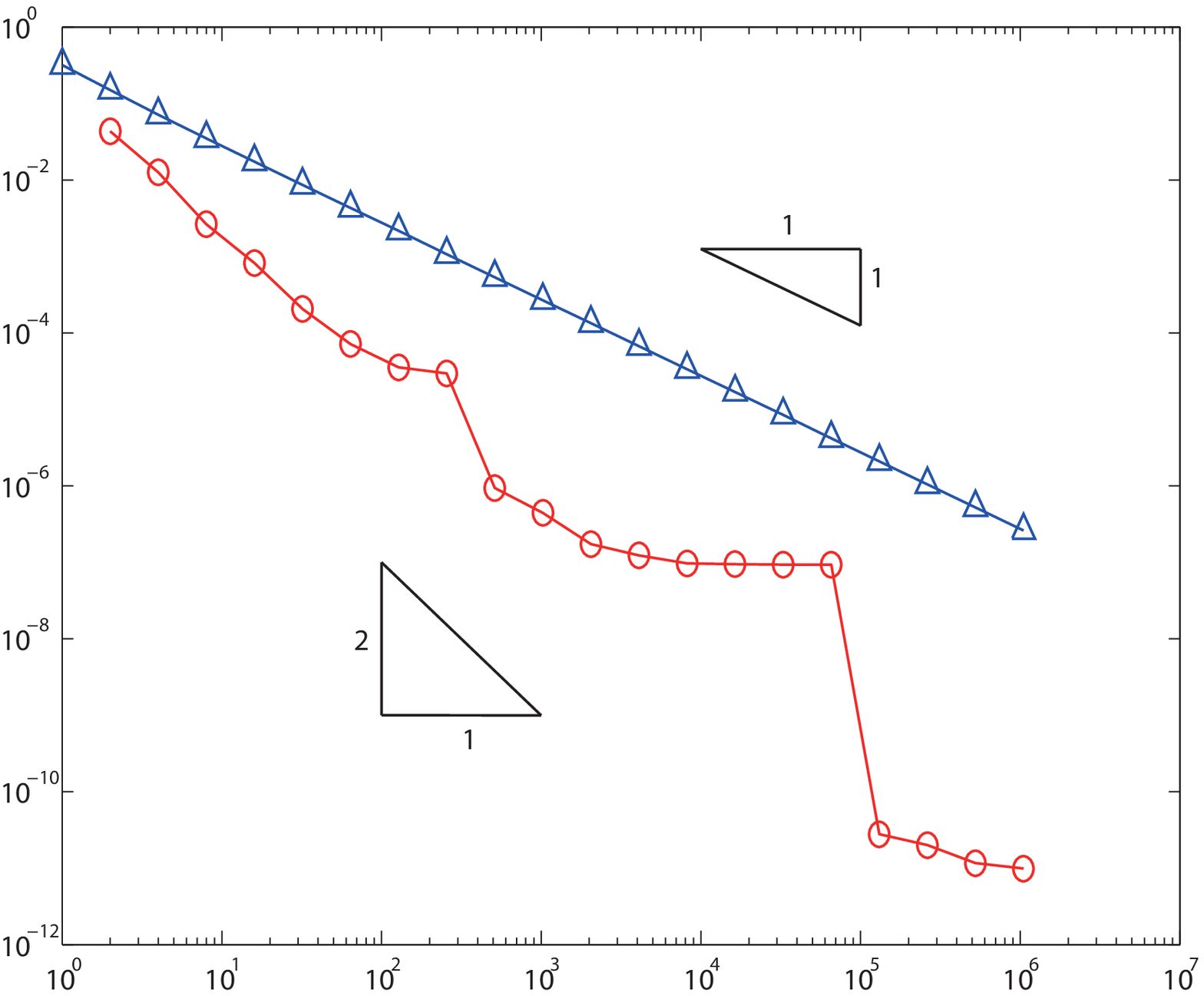}
\caption{Absolute integration error for $f_2$ (top) and $f_3$ (bottom) vs number of points. The lines marked by circle and triangle represent the errors using QMC rules based on a Sobol' point set with and without dyadic antithetics, respectively. (Left top: $s=100, \omega=0.5$, right top: $s=100, \omega=0.1$, left bottom: $ s=100, \omega=0.5$, right bottom: $s=100, \omega=0.1$)}
\label{fig:2}
\end{center}
\end{figure}

\end{document}